\documentclass[12pt,reqno,sumlimits]{amsart}

 \usepackage{amssymb,amscd,amsmath,epsfig}

\usepackage{graphicx,type1cm,eso-pic}

\newtheorem{thm}{Theorem}[section]
\newtheorem{cor}[thm]{Corollary}
\newtheorem{lem}[thm]{Lemma}
\newtheorem{prop}[thm]{Proposition}
\theoremstyle{definition}
\newtheorem{defn}{Definition}[section]
\newtheorem{rem}{Remark}[section]

\newtheorem{exm}[thm]{Example}

\newcommand{\R}{{\mathbb R}}

\newcommand{\C}{{\mathbb C}}

\newcommand{\Z}{{\mathbb Z}}

\newcommand{\CP}{{\mathbb C}{\mathbb P}}

\newcommand{\g}{{\mathfrak  g}}

\newcommand{\calB}{{\mathcal B}}

\newcommand{\calE}{{\mathcal E}}

\newcommand{\calG}{{\mathcal G}}
\newcommand{\calH}{{\mathcal H}}

\newcommand{\calM}{{\mathcal M}}

\newcommand{\calR}{{\mathcal R}}

\newcommand{\calL}{{\mathcal L}}

\renewcommand{\to}{\longrightarrow}

\newcommand{\ev}{\operatorname{ev}}

\newcommand{\Diff}{\operatorname{Diff}}
\newcommand{\Tr}{\operatorname{Tr}}
\newcommand{\tr}{\operatorname{tr}}

\newsavebox{\savepar}

\numberwithin{equation}{section}

\newcounter{labelflag} 
\newcommand{\labelon}{\setcounter{labelflag}{1}}
\newcommand{\Label}[1]{
                       \ifnum\thelabelflag=1
                          \ifmmode
                             \makebox[0in][l]{\qquad\fbox{\rm#1}}
                          \else
                             \marginpar{\vspace{0.7\baselineskip}
                                        \hspace{-1.1\textwidth}
                                        \fbox{\rm#1}}
                          \fi
                       \fi
                       \label{#1}
                      }
\labelon

  \newcommand{\BbR}{{\mathbb R}}

 \newcommand{\BbC}{{\mathbb C}}
 \newcommand{\BbZ}{{\mathbb Z}}
 \newcommand{\pdo}{\Psi{\rm DO}}

 \newcommand{\dvol}{{\rm dvol}}

 \newcommand{\dg}{\dot\gamma}

 \newcommand{\ints}{\int_{S^1}}

 \newcommand{\dir}{\partial\kern-.570em /}
 \newcommand{\dire}{\partial\kern-.570em /{}^{\rm eq}}
 \newcommand{\GG}{\pdo_0^*}

 \newcommand{\diffm}{{\rm Diff}(M) }
 \newcommand{\diff}{{\rm Diff} }
 \newcommand{\maps}{{\rm Maps} }
 \newcommand{\kk}{2k-1 }
 
 \newcommand{\resw}{{\rm res}^{ W}}
 
 \newcommand{\be}{\overline\eta}
  \newcommand{\bn}{\overline\nabla}
  \newcommand{\bxi}{\overline\xi}
\newcommand{\br}{\overline R}
\newcommand{\bmk}{\overline{ M}_p}
\newcommand{\bg}{\overline g}
\newcommand{\bx}{\overline X}
\newcommand{\by}{\overline Y}

\newcommand{\bp}{\overline \Phi}
\newcommand{\xl}{X^L}
\newcommand{\yl}{Y^L}
\newcommand{\zl}{Z^L}
\newcommand{\la}{\langle}
\newcommand{\ra}{\rangle}
\newcommand{\sgn}{{\rm sgn}}

\newcommand{\eb}{\overline e}
\newcommand{\ebo}{\eb_{\sigma(1)}}
\newcommand{\ebt}{\eb_{\sigma(2)}}
\newcommand{\ebth}{\eb_{\sigma(3)}}
\newcommand{\ebf}{\eb_{\sigma(4)}}
\newcommand{\ebfi}{\eb_{\sigma(5)}}

\newcommand{\el}{e^L}
\newcommand{\elt}{e^L_{\sigma(2)}}
\newcommand{\elth}{e^L_{\sigma(3)}}
\newcommand{\elf}{e^L_{\sigma(4)}}
\newcommand{\elfi}{e^L_{\sigma(5)}}

\newcommand{\eet}{e_{\sigma(2)}}
\newcommand{\eeth}{e_{\sigma(3)}}
\newcommand{\eef}{e_{\sigma(4)}}
\newcommand{\eefi}{e_{\sigma(5)}}

\newcommand{\aoa}{A_{1a}}
\newcommand{\aob}{A_{1b}}
\newcommand{\aoc}{A_{1c}}

\newcommand{\ssum}{ \sum_{\sigma(1) = 1}\sgn(\sigma)}
\newcommand{\vol}{{\rm vol}}
\newcommand{\mab}{\overline {M}_{(a,b)} }
\newcommand{\mabp}{\overline {M}_{p(a,b)} }
\newcommand{\lp}{\Lambda^2_+}
\newcommand{\lm}{\Lambda^2_-}
\newcommand{\rpp}{R_{++}}
\newcommand{\rmm}{R_{--}}
\newcommand{\bpa}{\overline{M}_{p\vec a}}
\newcommand{\ba}{\overline{M}_{\vec a}}
\newcommand{\tid}{\rm Id}

\newcommand{\chk}{ch^W_{[2k]}}

\begin{document}

\newpage\null\vskip-4em
  \noindent\scriptsize{}
  \scriptsize{
  \textsc{AMS Mathematics Subject
  Classification Numbers: 58J28, 58J40.
} }
\vskip 0.5 in

\normalsize

\title{The Geometry of Loop Spaces II:  Characteristic Classes}
\author[Y. Maeda]{Yoshiaki Maeda}
\thanks{Y. Maeda is supported in part by JSPS KAKENHI no. 23340018 and no. 25610015, and  by JSPS Core-to-Core program on Foundation of a Global Research Cooperative Center in Mathematics Focused on Number Theory and Geometry }
\address{Department of Mathematics\\
Keio University}
\email{maeda@math.keio.ac.jp}
\author[S. Rosenberg]{Steven Rosenberg}
\address{Department of Mathematics and Statistics\\
  Boston University}
\email{sr@math.bu.edu}
\author[F. Torres-Ardila]{Fabi\'an Torres-Ardila}
\address{Center of Science and Mathematics in Context\\
  University of Massachusetts Boston}
\email{fabian.torres-ardila@umb.edu}

\begin{abstract}    In this follow-up paper to \cite{MRT3}, we prove that $|\pi_1(\diff(\overline M))| = \infty$ for the total space $\overline M$ of  circle bundles associated to high multiples of a K\"ahler class 
over integral K\"ahler surfaces.  To detect nontrivial elements of $\pi_1(\diff(\overline M))$, we develop 
a theory of  Chern-Simons classes
  $CS_{2k-1}^W(\nabla^0, \nabla^1) 
 \in H^{2k-1}(LM^{2k-1};\R)$ on the loop space $LM$ of a Riemannian manifold.
These classes use the Wodzicki residue of the connection and curvature forms of 
the $L^2$ connection $\nabla_0$ and the Sobolev parameter $s=1$ connection $\nabla^1$ on $LM$, as these forms
take values in pseudodifferential operators.
 \end{abstract}

\maketitle
\definecolor{ColorEdit}{named}{red}

\noindent {\bf Keywords:}  Loop spaces, characteristic classes, pseudodifferential operators, Wodzicki residue, diffeomorphism groups

\bigskip

\centerline{Dedicated to the memory of Prof. Shoshichi Kobayashi}

\bigskip

\noindent \textcolor{red}{Warning: The main results in this paper are valid for the isometry groups of $\overline M$, not the diffeomorphism groups.  Corrections are in {\em The Geometry of Loop Spaces II: Corrections}, arXiv:2405.00651.}

\section{Introduction}

The loop space $LM$ of a Riemannian manifold $(M,g)$ admits a family of Riemannian metrics $g^s$ depending on $g$ and a Sobolev space parameter $s \geq 0.$  In \cite{MRT3}, we calculated the Levi-Civita connections associated to $g^s.$  In this paper, we develop a theory of characteristic classes on 
the tangent bundle $TLM$ associated to these metrics.  It turns out  that the Pontrjagin classes  are trivial, but the associated Chern-Simons classes are nontrivial in $H^{\rm odd}(LM,\R)$
in general.   As the main application, we use a specific 
Chern-Simons class  to prove that $|\pi_1(\diff(\overline M))| = \infty$ for the total space $\overline M$ of  circle bundles associated to high multiples of a K\"ahler class 
over integral K\"ahler surfaces.  These circle bundles include several topologically distinct infinite families of $5$-manifolds, including $S^2\times S^3$ and some lens spaces.

To develop a Chern-Weil theory for these connections, we need invariant polynomials on the Lie algebra of the structure group of the Levi-Civita connections, which by \cite{MRT3} 
is the group of invertible zeroth order pseudodifferential operators ($\pdo$s). The naive choice is
the standard polynomials $\Tr(\Omega^k)$ of the curvature $\Omega = \Omega^s$,
where Tr is the operator trace.  However, $\Omega^k$ is zeroth order
and hence not trace class, and in any case the operator trace
is impossible to compute in general.  Instead, as in \cite{P-R2} we use the 
Wodzicki residue, the only trace on the full
algebra of $\pdo$s.  
Following Chern-Simons
 \cite{C-S} as much as possible, we build a theory of
Wodzicki-Chern-Simons (WCS) classes, which gives classes in $H^{2k-1}(LM^{2k-1},\R)$ associated to 
the Wodzicki-Chern character of $TLM$.

 There are several differences from  finite
dimensional Chern-Simons theory. The absence of a Narasimhan-Ramanan universal connection
theorem means that we do not have a theory of differential characteristic classes
(see e.g. \cite{FL}).  Moreover, the use of the Wodzicki residue allows to define WCS classes associated to the Chern character but not to the Chern classes.  On the positive side,
since we have a family of connections on $LM$, we can define the relative  $\R$-valued, not just 
the absolute $\R/\Z$-valued, WCS classes associated to a Riemannian metric on $M$. Finally,  if 
dim$(M)=3$, the
 WCS form vanishes, a result without a finite dimensional analogue.

 In contrast to the operator trace, the Wodzicki residue is locally
 computable, so we can write explicit expressions for the WCS classes.
In particular, we can see how the WCS classes depend on the Sobolev parameter $s$, and 
hence define   ``regularized" or $s$-independent WCS classes.  
 As the main application, we compute the integral of $CS_5^W\in H^5(L\overline M, \R)$ over a specific $5$-cycle  for $\overline M$ as above.
 This leads to the results on $\pi_1(\diff(\overline M)).$

For related results on characteristic classes on infinite rank bundles with a group of $\pdo$s as structure group, see 
 \cite{lrst, paycha, P-R2}. For the important case of loop groups, see \cite{freed}.  For the general theory of bundles with $\pdo$s or Fourier integral operators as structure group, see \cite{MM}.

\medskip

This paper is organized as follows. In \S2, we review  finite dimensional Chern-Weil and 
Chern-Simons theory, and use the Wodzicki residue to define residue Chern character classes (RCC) and residue or Wodzicki-Chern-Simons (WCS) classes (Definition \ref{def:WCS}) on $LM$.  We prove the necessary vanishing
of the RCC classes for mapping spaces (and in particular for $LM$) in Proposition
\ref{prop:maps}.  In Theorem \ref{thm:5.5}, we give the explicit local expression for the
 WCS class $CS_{2k-1}^W(g)\in H^{2k-1}(LM^{2k-1})$.  We then define the regularized or $s$-independent WCS class.  
In Theorem \ref{WCSvan}, we give the vanishing result
for the WCS class in dimension $3$. 

In particular, the WCS class which is the analogue of the classical dimension three Chern-Simons class vanishes on loop spaces of 
 $3$-manifolds, so we look for nontrivial examples on $5$-manifolds.
 In \S\ref{dimfive}, we consider the total spaces $\overline M_p$ of circle bundles over integral K\"ahler surfaces associated to  integer multiples $p\omega$ of an integral  K\"ahler class $\omega$.  These $5$-manifolds admit a Sasakian structure, which makes the computation of $CS_5^W$ tractable, if unpleasant.  
 The $S^1$-action given by rotation in the fibers of $\overline M_p$ gives both a $5$-cycle $[a]$ in $L\overline 
 M_p$ and an element of $\pi_1(\diff(\overline M_p)).$  We compute that $\int_{[a]}CS^W_5\neq 0$ for
 $|p| \gg 0$, which  implies $|\pi_1(\diff(\overline M_p))| = \infty$ (Thm. \ref{bigthm}).  In the case of
 specific K\"ahler surfaces such as $T^4$, $\CP^2$, $S^2\times S^2$, and K3 surfaces, we can give sharp estimates on $p.$  We give an alternative explicit computer computation to check that $|\pi_1(\diff(S^2\times S^3)))| = \infty. $ 
 
 Appendix A proves Lemma 3.7, and Appendix B discusses the topology of $LM$ and $\diffm.$

\medskip
We  thank Kaoru Ono  for pointing out an error in  a previous version of 
the paper, and thank Christian Becker, Alan Hatcher, Danny Ruberman, and the referee for helpful comments. 
\medskip

\noindent {\it Notation:}
 $H^*$ always refers to de Rham cohomology for complex valued forms.  By \cite{beggs}, $H^*(LM)\simeq H^*_{\rm sing}(LM,\C).$  Our convention for the curvature tensor is 
 $R(\partial_j, \partial_k)\partial_b = R_{jkb}^{\ \ \ a}\partial_a$.

\section{{\bf Chern-Simons Classes on Loop Spaces}}\label{CSCLS}

We  review Chern-Weil and Chern-Simons theory for
finite dimensional vector bundles \cite[Ch. 4]{Rbook}.

\subsection{{\bf Chern-Weil  and Chern-Simons Theory for Finite Dimensional
  Bundles} }

Let $G$ be a finite dimensional linear Lie group with Lie algebra $\g$, let $Q\to M$ be a 
principal $G$-bundle over a manifold $M$, and let
 $ Q\times_\rho G\to M$ be the vector bundle over a manifold $M$ associated to a representation 
 $\rho$ of $G$.
Set $
 \g^k=\g^{\otimes k}$ and let
\begin{equation*}I^k(G)
= \{P:\g^k\to \C\ | P\ \text{symmetric,
  multilinear, Ad-invariant}\}
\end{equation*}
be the degree $k$ Ad-invariant polynomials on $\g.$

\begin{rem}
For  $G = U(n),$ resp. $O(n)$, $I^k(G)$ is generated by the polarization of
the Newton polynomials $P_k(A) = \Tr(A^k)$, resp. $\Tr(A^{2k})$, where $\Tr$ is the usual trace on finite
dimensional matrices.
\end{rem}

For $\phi\in\Lambda^\ell(M,\g^k)$, $P\in I^k(G)$, set
$P(\phi)=P\circ \phi\in\Lambda^\ell(M)$.

\begin{thm}[The Chern-Weil Homomorphism]
\label{previous}
Let $\nabla$ be a connection on $F\to M$ with curvature $\Omega\in
\Lambda^2(M,Q\times_\rho \g)$. For $P\in I^k(G)$, $P(\Omega)$ is a closed
  $2k$-form on $M$, and so
determines a de Rham cohomology class $[P(\Omega)]\in H^{2k}(M).$
The Chern-Weil map
\begin{equation*}
\oplus_{k}I^k(G)\to H^{*}(M), \ P\mapsto [P(\Omega)]
\end{equation*}
is a well-defined algebra homomorphism, and in particular is independent of the choice of
connection on $F$.
\end{thm}

$[P(\Omega)]$ is
 the {\it characteristic class} of $P$. For example, the characteristic class
 associated to  $\frac{1}{k!}\Tr(A^k)$ is the k${}^{\rm th}$ component of the Chern character of $F$.

Part of the theorem's content
is that for any two connections on $F$,
$P(\Omega_1) - P(\Omega_0) = 
dCS_P(\nabla_1,\nabla_0)$  
for some odd form $CS_P(\nabla_1, \nabla_0)$.  Explicitly, for $\nabla_i = d+ \omega_i, i = 0,1, $ locally, 
\begin{equation}\label{5.1}
CS_P(\nabla_1,\nabla_0) = k\int_0^1 P(\omega_1-\omega_0,\overbrace{\Omega_t,...,\Omega_t}^{k-1})
\ dt
\end{equation}
where 
$$\omega_t = t\omega_0+(1-t)\omega_1,\ \ \Omega_t = d\omega_t+\omega_t\wedge\omega_t$$ 
 \cite{Rbook}.  
 Note that $ \omega_1-\omega_0, \Omega_t$ are globally defined $\g$-valued forms on $M$.
 $CS_P(\nabla_1,\nabla_0)$ is the {\it transgression} or {\it Chern-Simons} form associated to $P$, relative
 to  $\nabla_1,\nabla_0$.

\begin{rem}  We relate this construction to the  absolute version of this construction on principal $G$-bundles $\tilde F\stackrel{\pi}{\to} M$ in \cite{C-S}. 
$\pi^*F\to F$ is trivial, so it has 
 a flat connection $\theta_1$ on $\pi^*F$
with respect to a fixed trivialization.
Let $\theta_1$ also
denote the connection $\chi^*\theta_1$ on $F$, 
where $\chi$ is the global section of $\pi^*F.$  For any other connection $\theta = \theta_0$ on $F$,  $\theta_t = t\theta_0, \Omega_t = t\Omega_0 + (t^2-t)\theta_0\wedge \theta_0$. 
We obtain the formulas for the transgression form $TP(\theta)$
on $F$: for 
\begin{equation}\label{eq:ChernSimons}
\phi_t =t\Omega_1+\frac{1}{2}(t^2-t)[\theta,\theta],\ \ 
TP(\theta)=k\int_0^1 P(\theta\wedge \phi^{k-1}_t)dt,
\end{equation}
$dTP(\theta)=P(\Omega_1)\in \Lambda^{2k}(F)$
\cite{C-S}.  Assume the image of $P$ under the abstract Chern-Weil homomorphism
$\oplus_{k}I^k(G)\to H^{*}(BG)$ lies in the image of $H^*(BG,\Z).$
Then
$TP(\Omega_1)$ pushes down to an $\BbR/\BbZ$-class on $M$,
the absolute Chern-Simons class.
\end{rem}

 \subsection{{\bf Chern-Weil and Chern-Simons Theory for $\pdo_0^*$-Bundles}}

Fix $s>0$.  Let $LM = L^{s'}M$ be the space of loops from $S^1\to M$ in a fixed Sobolev class 
$s'\gg s.$  This restriction makes $LM$ into a paracompact Hilbert manifold, which is technically easier to work with than the Fr\'echet manifold $L^FM$ of smooth loops.  By the discussion in Appendix B,
$L^{s'}M$ and $L^FM$ are homotopy equivalent, so the main results hold in both topologies.

The Riemannian metric $g^s$ on $LM$ associated to the Riemannian metric $g$ on $M$ is defined at a loop $\gamma\in LM$ by
$$g^s(X,Y)_\gamma = \frac{1}{2\pi}\int_{S^1} g(X_{\gamma(\theta)}, (1+\Delta)^s Y_{\gamma(\theta})_{\gamma(\theta)}d\theta,$$
where $\theta\mapsto X_{\gamma(\theta)}, Y_{\gamma(\theta)}\in T_{\gamma(\theta)}M
$ are elements
of $T_\gamma LM=\Gamma^{s-1}(\gamma^*TM)$ (sections of Sobolev class $s-1$), $\Delta = \left(\frac{D}{d\theta}\right)^*\frac{D}{d\theta}$ is the Laplacian built from covariant differentiation along $\gamma$, and $(1+\Delta)^s $ is the associated pseudodifferential operator
($\pdo$).  We call $g^s$ the Sobolev $s$-metric.
The connection one-form and curvature two-form of the associated Levi-Civita connection $\nabla_s$ take values in the algebra $\pdo_{\leq 0}$  of nonpositive order $\pdo$s
acting on sections of the trivial $\R^n$ bundle $\calR^n$ over 
$S^1$, where $n = {\rm dim} \ M.$  For $s=0$, we have the usual $L^2$ metric and connection. 
See \cite[\S2]{MRT3} for details. 

Even though the structure group of $TLM$ is the gauge group $\calG$ of $\calR^n$, the connection and curvature forms take values in an algebra larger than the Lie algebra ${\rm Lie}(\calG)$. 
  Thus we should extend the gauge group to the larger group $\GG(\calR^n)$ of zeroth order invertible $\pdo$s acting on sections of $\calR^n$, since
  ${\rm Lie}(\GG) = \pdo_{\leq 0}.$ 
 
 In general, let $\calE\to \calM$ be an infinite rank bundle over a paracompact Banach manifold $\calM$, with the fiber of $\calE$ modeled on a fixed Sobolev class of sections of a finite rank complex bundle $E\to N$, and with structure group $\GG(E)$.  
 For such $\GG$-bundles,
 we can produce
 primary and secondary characteristic classes 
 once we choose Ad-invariant polynomials of $\pdo_{\leq 0}.$

 Since the adjoint action of $\GG$ on $\pdo_{\leq 0}$ is by conjugation,  any trace $T$ on $\pdo_{\leq 0}$ will 
 produce invariant polynomials $A\mapsto T(A^k)$, just as for $\mathfrak u(n).$ 
 These traces were classified  in \cite{lesch-neira, paycha-lescure}, although there are slight variants
in our special case $N= S^1$ \cite{ponge}.  Roughly speaking, the traces fall into two classes, the leading order symbol trace \cite{P-R2} and the Wodzicki residue.  In this paper,
we consider only the Wodzicki residue, and refer to \cite{LMRT,lrst} for the leading order symbol
trace.

Recall that a classical $\pdo$ $P$ acting on sections of $E\to N$ has an order $\alpha\in \R$ and a symbol expansion $\sigma^P(x,\xi) 
\sim \sum_{k=0}^\infty \sigma^P_{\alpha-k}(x,\xi),$ where $x\in N, \xi\in T^*_xN$, and $\sigma^P_{\alpha-k}(x,\xi)$ is homogeneous of degree $\alpha-k$ in $\xi.$  For  $(x,\xi)$ fixed , $\sigma^P(x,\xi),
 \sigma_{\alpha-k}^P(x,\xi) \in {\rm End}(E_x).$ 
 We note that real $\pdo$s on complexified bundles will have symbols which are real endomorphisms.

The Wodzicki residue of $P$ is
$$\resw(P) = \frac{1}{(2\pi)^n} \int_{S^*N} \tr \sigma_{-n}^P(x,\xi) d\xi\ dx,$$
where $S^*N$ is the unit cosphere bundle over $N$ with respect to a fixed Riemannian metric, and dim$(N)=n.$ It is nontrivial that $\resw$ is independent of coordinates and defines a trace: $\resw[P,Q] = 0.$

We will  restrict attention to the generating invariant polynomials\\
 $P_k(A_1,\ldots,A_k)$, the
symmetrization of
$\frac{1}{k!}\Tr(A_1\cdot\ldots\cdot A_k)$, whose corresponding characteristic class $\frac{1}{k!}[\Tr(\Omega^k)]$ is the degree $2k$ component of the Chern character.  
We
only consider $\mathcal E = TLM$, which here denotes the complexified tangent bundle.
Thus the corresponding Chern classes are really Pontrjagin classes.  For $TLM$, 
$E\to N$ is the trivial complex bundle $S^1\times \C^n\to S^1$.

\begin{defn}  \label{def:WCS} 
(i) The k${}^{\rm th}$ {\it residue Chern character (RCC) form} of a $\pdo_0^*$-connection
$\nabla$ on $TLM$ with curvature $\Omega$ is
\begin{equation}\label{5.1a}
\chk(\Omega)(\gamma) =
 \int_{S^*S^1}\tr\sigma_{-1}(\Omega^{k}) \ d\xi  dx.
\end{equation}
As above, for each $\gamma\in LM$,
$\sigma_{-1}(\Omega^k)$ is a $2k$-form with values in  endomorphisms
 of a trivial bundle
over $S^*S^1$, so $\chk \in \Lambda^{2k}(LM, \C).$ For convenience, we omit the usual $\frac{1}{k!}$ factor from the 
RCC form, and the usual constant in the Wodzicki residue.

(ii) The k${}^{\rm th}$ {\it Wodzicki-Chern-Simons (WCS) form} of two $\pdo_0^*$-connections 
$\nabla_0,\nabla_1$ on $TLM$ is
\begin{eqnarray}\label{5.22}
CS^W_{2k-1}(\nabla_1,\nabla_0) &=&k \int_0^1 dt \int_{S^*S^1}\tr\sigma_{-1}((\omega_1-\omega_0)\wedge 
(\Omega_t)^{k-1})\ d\xi dx\\ 
&=&k \int_0^1 {\rm res}^{W} 
[(\omega_1-\omega_0)\wedge 
(\Omega_t)^{k-1}]\ dt\in \Lambda^{2k-1}(LM,\C).\nonumber
\end{eqnarray}

(iii) The  k${}^{\rm th}$ {\it Wodzicki-Chern-Simons form} associated to a Riemannian metric 
$g$ 
on $M$, denoted $CS^W_{2k-1}(g)$,  is $CS^W_{2k-1}(\nabla_1,\nabla_0)\in \Lambda^{2k-1}(LM,\R)
$, where $\nabla_0, \nabla_1$ refer to the 
$L^2$ and Sobolev $s=1$ Levi-Civita connections on $LM$, respectively.

(iv) The
k${}^{\rm th}$ {\it residue Chern class form} $c_k^W(\Omega)$ is defined by the appropriate algebraic combinations of the residue Chern character given by the Newton polynomials; {\it e.g.}, $c_1^W = ch_{[2]}^W,$
$c_2^W(\Omega) = \frac{1}{2}\left( ch_{[2]}^W(\Omega)\wedge ch_{[2]}^W(\Omega) - ch_{[4]}^W(\Omega\right).$
 Let $P= \sum a^KP_K$ be a polynomial in the $U(n)$-invariant polynomials $P_k: A\mapsto \frac{1}{k!}
\Tr(A^k)$, with $K = (k_1,\ldots, k_r)$ and $P_K = P_{k_1}\cdot\ldots\cdot P_{k_r}$. Define the
{\it residue characteristic form} by
$$c_P^W(\Omega) = \sum a^K c_{k_1}^W(\Omega)\wedge \ldots\wedge c_{k_r}^W(\Omega).$$

\end{defn}

\medskip

As in finite dimensions, $c_P^W(\Omega)$ is a closed $2k$-form, with de Rham cohomology
class $c_P(LM)$
 independent of $\nabla$, as 
 $$c_P^W(\Omega_1)  - c_P^W(\Omega_0) =
dCS^W_{P}(\nabla_1,\nabla_0).$$
Here $CS^W_P$ is defined as in (\ref{5.1}), with $P$ replaced with $\resw.$
  In particular, in our notation
$$\chk (\Omega_1)  - \chk (\Omega_0) =
dCS^W_{2k-1}(\nabla_1,\nabla_0).$$
 
\begin{rem}  It is an interesting question to determine all  $\pdo_0^*$-invariant polynomials on 
$\pdo_{\leq 0}.$  As above, $U(n)$-invariant polynomials combine with the Wodzicki residue 
(or the other traces on $\pdo_{\leq 0}$) to give $\pdo_0^*$-polynomials,
but there may be others.  
\end{rem}

We now prove that $TLM$  and more generally the tangent bundle to mapping spaces
Maps$(N,M)$,  with $N$ a closed manifold,
have vanishing residue Chern classes.  As above, we take a Sobolev topology on Maps$(N,M)$ for some
large Sobolev parameter.
We denote the de Rham class of $c_{P}^W(\Omega)$ for a connection on $\mathcal E$ by
$c_{P}(\mathcal E),$  and for the special case $\mathcal E = T{\rm Maps}(N,M)\otimes \C$ by $c^W_{P}(
{\rm Maps}(N,M)).$

\begin{prop} \label{prop:maps} Let $N, M$ be closed manifolds, and let {\rm Maps}${}_f(N,M)$ denote
the component of a fixed $f:N\to M$.  Then the residue characteristic  classes\\
$c_{P}^W({\rm Maps}_f(N,M)) $  vanish.
\end{prop}

\begin{proof}
For $TLM$, the $L^2$ connection 
has curvature $\Omega$ which is a multiplication operator/pointwise endomorphism \cite[Lemma 2.1]{MRT3}.  Thus $\sigma_{-1}(\Omega)$ and hence $\sigma_{-1}(\Omega^{i})$ are zero, 
so the residue characteristic forms $c_{P}(\Omega)$ vanish.

For $n\in N$ and $h:N\to M$,
let $\ev_n: {\rm Maps}_f(N,M)$ be $\ev_n(h) = h(n).$ 
 Then $D_XY(h)(n) \stackrel{\rm def}{=} 
 (\ev_n^*\nabla^{LC,M})_XY(h)(n)$ is the $L^2$ Levi-Civita connection on 
Maps$(N,M).$
As in \cite[Lemma 2.1]{MRT3},
the curvature of $D$ is  
a multiplication operator.  Details are left to the reader.
\end{proof}

In finite dimensions,  characteristic classes are topological obstructions to the
reduction of the structure group, and geometric obstructions to the existence
of a flat connection.  
RCC classes for $\pdo_0^*$-bundles 
are also topological and geometric obstructions, but
the geometric information is a little more refined due to the grading on the
Lie algebra 
 $\pdo_{\leq 0}$.

\begin{prop}
 Let $\calE\to\calB$ be 
 a $\GG$-bundle, for
  $\GG$ acting on
$E\to N^n$.  
If $\calE$ admits a reduction to the gauge group $\calG(E)$, then
  $c_P^W(\calE)  =   0$ for all $P$.
  If $\calE$ admits a 
   $\GG$-connection whose
  curvature has order $-k$, then $
  c_{\ell}(\calE) =0$ for $\ell \geq [n/k].$
 \end{prop}

\begin{proof}  If the structure group of  $\calE$ reduces to the gauge
  group, there exists a connection one-form
  with values in Lie$(\calG) = {\rm End}(E)$, the Lie algebra of multiplication
  operators.  Thus the Wodzicki residue of powers of the curvature vanishes,
  so the residue Chern character classes vanish.
For the second statement, the order of $\Omega^\ell$  is less than
$-n$ for $\ell \geq [n/k]$,  so the Wodzicki residue
  vanishes in this range. 
  \end{proof}
  
  However, we do not have examples of nontrivial RCC classes; cf.~\cite{lrst}, where it is 
  conjectured that these classes always vanish.

 \subsection{Local calculations}   

We now use local symbol calculations to compute WCS forms.   
\medskip

\noindent {\bf Notation and Conventions:} For curvature conventions for $M$, we set
$$\Omega^M(\partial_k, \partial_j)^{\ a}_b = R_{kjb}^{\ \ \ a} = R(\partial_k,\partial_j)^{\ a}_b, 
\ R(\partial_k,\partial_j,\partial_b,\partial_a) = \langle R(\partial_k,\partial_j)\partial_b,\partial_a\rangle
= R_{kjba},$$
in agreement with \cite{MRT3}. Our convention for wedge product is 
$\omega\wedge\eta = \frac{(k+\ell)!}{k!\ell!}{\rm Alt}(\omega\otimes \eta)$ for $\omega\in \Lambda^k, \eta
\in\Lambda^\ell.$   
\medskip

For completeness, we restate some results from \cite[Appendix A]{MRT3} about the symbols of the connections $\nabla_0, \nabla_1$ on $TLM.$ 

\begin{lem}\label{MRT3lem} Fix a metric $g$ on $M$ with connection form $\omega^M$ and curvature tensor $R = R^M$.
Let $\omega^0, \Omega^0$ denote the connection and curvature forms for the $L^2$ ($s=0$) metric on 
$LM$, let $\omega^1, \Omega^1$ denote  the connection and curvature forms for the $s=1$ metric, and let $\omega^M$ be the connection one-form for the Levi-Civita connection on $M$.
Fix a loop $\gamma\in LM.$  At $(\theta, \xi)\in T^*S^1$, 

(i) 
$ (\omega^0_X)^a_b = (\omega^M_X)^a_b$, where $\omega_X^M$ is computed at  $\gamma(\theta).$  $\sigma_0(\Omega^0(X,Y))^a_b= R(X,Y)^a_b = R_{cdb}^{\ \ \ a}X^cY^d.$  $\omega^0, \Omega^0$ are multiplication operators, so $\sigma_i(\omega^0) = \sigma
_i(\Omega^0) = 0$ for $i<0.$ 

(ii) $\sigma_0(\omega^1_X)^a_b =  (\omega^M_X)^a_b$.
\begin{eqnarray*}
\frac{1}{i|\xi|^{-2}\xi}\sigma_{-1}(\omega^1_X) &=& \frac{1}{2}(-2R(X,\dg)
-R(\cdot,\dg)X + R(X,\cdot)\dg).
\end{eqnarray*}


\end{lem}

With this Lemma, the WCS forms become tractable.

\begin{prop} \label{2.4}  Let $\sigma$ be a permutation of $\{1,\ldots,2k-1\}.$  Then
\begin{eqnarray}\label{5.4}
\lefteqn{CS^W_{2k-1}(g)(X_1,...,X_{2k-1}) }\\
&=&
 \frac{k}{2^{k-1}}     
\sum_{\sigma} {\rm sgn}(\sigma) \int_{S^1}\tr [
(-2R(X_{\sigma(1)},\dg)
-R(\cdot,\dg)X_{\sigma(1)} + R(X_{\sigma(1)},\cdot)\dg) \nonumber\\
&&\qquad 
\cdot (\Omega^M)^{k-1}(X_{\sigma(2)},\ldots,X_{\sigma(2k-1)} )].\nonumber
\end{eqnarray}
\end{prop}

\begin{proof}  By Lemma~\ref{MRT3lem}(i), 
$\sigma_0((\omega_1-\omega_0)_X)=0.$
Thus 
\begin{equation}\label{cswint}
CS^W_{2k-1}(g) = k \int_0^1dt \int_{S^*S^1}\tr[\sigma_{-1}(\omega_1-\omega_0)\wedge (\sigma_0(\Omega_t))^{k-1}]\ d\xi dx.
\end{equation}
Moreover,
\begin{eqnarray*}\sigma_0(\Omega_t) &=& td(\sigma_0(\omega_0)) + (1-t)d(\sigma_0(\omega_1)) \\
&&\qquad + 
(t\sigma_0(\omega_0) + (1-t)\sigma_0(\omega_1))\wedge (t\sigma_0(\omega_0) + (1-t)\sigma_0(\omega_1))\\
&=& d\omega^M + \omega^M\wedge \omega^M\\
&=& \Omega^M.
\end{eqnarray*}
Therefore,
\begin{equation}\label{5.3}
CS^W_{2k-1}(g) = k \int_0^1 dt\int_{S^*S^1}\tr [\sigma_{-1}(\omega_1)
\wedge (\Omega^M)^{k-1}]\ d\xi dx,
\end{equation}
since $\sigma_{-1}(\omega_0) = 0.$  We can drop the integral over $t$. 
The integral over the $\xi$ variable contributes a factor of $2$: the integrand has
a factor of $|\xi|^{-2}\xi$, which equals $\pm 1$ on the two components of $S^*S^1$.
Since the fiber of $S^*S^1$ at a fixed $\theta$ consists of two points 
with opposite orientation, the ``integral" over each fiber is $1-(-1) = 2.$ 
Thus
\begin{eqnarray}\label{5.4a}  \lefteqn{
CS^W_{2k-1}(g)(X_1,...X_{2k-1})    }  \\
&=&   
\frac{2k}{2^{k-1}}    
\cdot\frac{1}{2} \sum_\sigma {\rm sgn}(\sigma)  \int_{S^1}\tr[
(-2R(X_{\sigma(1)},\dg)
-R(\cdot,\dg)X_{\sigma(1)} + R(X_{\sigma(1)},\cdot)\dg)\nonumber\\
&&\qquad
\cdot (\Omega^M)^{k-1}(X_{\sigma(2)},\ldots,X_{\sigma(2k-1)} )],\nonumber
\end{eqnarray}
by  Lemma~\ref{MRT3lem}(ii).
\end{proof}

This produces odd classes in the de Rham cohomology of the loop space of an odd
dimensional manifold.

\begin{thm}\label{thm:5.5}
 (i)  Let dim$(M) = 2k-1$.
 Then $\chk(\Omega) \equiv 0$ for any $\pdo_0^*$-connection $\nabla$ on
 $TLM.$  Thus the k${}^{\it th}$ Wodzicki-Chern-Simons form 
 $CS^W_{2k-1}(\nabla_1,\nabla_0)$ is closed and defines a 
 class $[CS^W_{2k-1}(\nabla_1,\nabla_0)]\in H^{2k-1}(LM).$  In particular, we can
 define $[CS^W_{2k-1}(g)]\in H^{2k-1}(LM)$ for a Riemannian metric $g$ on $M$.

(ii)  For dim$(M) = 2k-1$,  $CS^W_{2k-1}(g)$
simplifies to 
 \begin{eqnarray}\label{csg}
\lefteqn{CS^W_{2k-1}(g)(X_1,...,X_{2k-1}) }\\
&=&
\frac{k}{2^{k-2}} 
\sum_{\sigma} {\rm sgn}(\sigma) \int_{S^1}\tr[
 (R(X_{\sigma(1)},\cdot)\dg)
 (\Omega^M)^{k-1}(X_{\sigma(2)},\ldots,X_{\sigma(2k-1)} )].\nonumber
\end{eqnarray}

 \end{thm}
 
 \begin{proof} (i) Let $\Omega$ be the curvature of $\nabla.$
 $\chk(\Omega)(X_1,\dots, X_{2k})(\gamma)$ is an integral of the pointwise expression which is
a $2k$-form on $M$, and hence vanishes.

 (ii) Denote the three terms on the right hand side of the second line of (\ref{5.4}) by (I), (II), (III), resp. Since
 \begin{eqnarray}\label{insert3} \tr[R(X_{1},\dg)
\cdot (\Omega^M)^{k-1}(X_{2},..X_{2k-1})] &=& 
[i_{\dg}\tr(\Omega^{k})](X_1,...X_{2k-1}) \\
&=& \tr(\Omega^k)(\dg, X_1,\ldots,X_{2k-1}),\nonumber
\end{eqnarray}
(I) vanishes on a $(2k-1)$-manifold.  
By the Bianchi identity,  
(II) equals  
$$[R(\dot\gamma, X_{\sigma(1)})\cdot + R(X_{\sigma(1)}, \cdot)\dg]
(\Omega^M)^{k-1}(X_{\sigma(2)},..X_{\sigma(2k-1)} ).$$
   The first term is of type (I), so its contribution vanishes, and the second term equals (III).  Thus the right hand side of (\ref{5.4}) simplifies to 2(III).   

\end{proof}

\begin{rem}\label{nonclosed} (i) The argument in Thm.~\ref{thm:5.5}(i) fails for the general invariant polynomials in 
Defn.~\ref{def:WCS}(iv).  For expressions like $c_{k_1}^W(\Omega)\wedge\ldots\wedge c^W_{k_r}(\Omega)
(X_1,\ldots, X_{2k}), r >1,$ are not the integral of a $2k$-form around a loop, but are products of integrals of lower degree forms around this  loop.  In particular, we do not construct residue Chern-Simons  classes associated to Chern classes, as these are polynomials in the components of the Chern character and so involve wedging of forms.

(ii)  There are several variants to the construction of relative WCS classes.

(a) If we define the transgression form $Tc_k(\nabla)$ with the Wodzicki residue
replacing the trace in (\ref{eq:ChernSimons}), it is easy to check that $Tc_k(\nabla)$
involves $\sigma_{-1}(\Omega).$  For $\nabla$ the $L^2$ connection, this WCS class vanishes.  For $\nabla$ the Levi-Civita connection on $LM$ for the Sobolev $s$-metric, $s>1/2$, $\sigma_{-1}(\Omega)$ involves 
the covariant derivative of the curvature of $M$ (cf. \cite[Lemma A.2]{MRT3} for $s=1.$)  Thus the
relative WCS class is easier for computations than the absolute class $[Tc_k(\nabla)].$

(b) If we define $CS_k^W(g)$ using the Levi-Civita connection for the Sobolev $s$-metric instead of
the $s=1$ metric, the WCS class is simply multiplied by the artificial parameter $s$ by 
\cite[Lemma A.3]{MRT3}.
 Therefore setting $s=1$ is not only computationally convenient, it 
regularizes the WCS, in that it extracts the $s$-independent information.
This justifies the following definition:

\begin{defn}  \label{def:regularized}
The {\it regularized  $k^{th}$ WCS class} associated to a Riemannian metric 
$g$ on $M$ is $CS_k^{W, {\rm reg}}(g) \stackrel{\rm def}{=} 
CS_k^W(\nabla^{1},\nabla^0)$, where $\nabla^{1}$ is the Levi-Civita  connection for the Sobolev
$s=1$ metric,
and $\nabla^0$ is the $L^2$ Levi-Civita connection.  
\end{defn}  

\end{rem}

\bigskip

We conclude this section with a vanishing result that has no finite dimensional analogue.

 \begin{prop} \label{WCSvan} 
  The  WCS form $CS_{3}^W(g)$
   vanishes.
  \end{prop} 
  
 \begin{proof}  Fix a loop $\gamma$ and a parameter $\theta$, and let $\{e_r\}$ be an orthonormal frame at $\gamma(\theta).$   Since the integrand at $\theta$ vanishes if $\dg(\theta)=0$, we can assume that $\dg(\theta)$ is a multiple of $e_1.$  For simplicity, we assume $\dg(\theta) = e_1.$ 

 We may assume that $X_r = e_r$ in (\ref{csg}). A typical term in (\ref{csg}) is 
\begin{eqnarray*} I_\sigma &=&
 \sgn(\sigma) R(e_{\sigma(1)}, e_{r_1}, e_{1}, e_{r_k})
R(e_{\sigma(2)}, e_{\sigma(3)}, e_{r_2}, e_{r_1})
R(e_{\sigma(4)}, e_{\sigma(5)}, e_{r_3}, e_{r_2})   \\
&&\qquad \cdot\ldots
R(e_{\sigma(2k-2)}, e_{\sigma(2k-1)}, e_{r_k}, e_{r_{k-1}})\\
&=& \sgn(\sigma)R_{\sigma(1) r_1 1 r_k}R_{\sigma(2)\sigma(3) r_2r_1}R_{\sigma(4)\sigma(5) r_3r_2}
\cdot\ldots\cdot R_{\sigma(2k-2)\sigma(2k-1)r_kr_{k-1}}
\end{eqnarray*}  

For $k=2$, we have
\begin{eqnarray*} \frac{1}{2} CS^W_3(g)(e_1, e_2, e_3) = \sum_\sigma I_\sigma &=& \sum_\sigma \sgn(\sigma)R_{\sigma(1) s 1 t }R_{\sigma(2)\sigma(3) ts}\\
&=& R_{1s1t}R_{23ts} - R_{2s1t}R_{13ts} - R_{3s1t}R_{21ts}\\
&&\quad -R_{1s1t}R_{32ts}+ R_{2s1t}R_{31ts} +R_{3s1t}R_{12ts}\\
&=& 2[R_{1s1t}R_{23ts} + R_{2s1t}R_{31ts} +R_{3s1t}R_{12ts}]\\
&=& 2[(a)+(b)+(c)].
\end{eqnarray*}
Then 
\begin{eqnarray*}(a) &=& R_{1213}R_{2332}+R_{1312}R_{2323} = R_{1213}R_{2332}-R_{1213}R_{2332} = 0,\\
(b) &=& R_{2112}R_{3121}+ R_{2113}R_{3131} + R_{2312}R_{3123} = (\alpha_1)+(\beta_1)+(\gamma_1),\\
(c) &=& R_{3112}R_{1221}+R_{3113}R_{1231}+R_{3213}R_{1232}= (\alpha_2)+(\beta_2)+(\gamma_2).\end{eqnarray*}
Since
\begin{eqnarray*}(\alpha_1)+(\alpha_2) &=& R_{2112}R_{3121} +R_{3112}R_{1221} = 
R_{2112}R_{3121} - R_{2112}R_{3121} = 0,\\
(\beta_1) +(\beta_2)&=& R_{2113}R_{3131} + R_{3113}R_{1231} = R_{3131}R_{2113}-R_{3131}R_{2113} = 0,\\
(\gamma_1)+(\gamma_2) &=&R_{2312}R_{3123}+R_{3213}R_{1232} =
R_{2312}R_{3123}- R_{2312}R_{3123} = 0,
\end{eqnarray*}
we obtain $CS^W_3(g) = 0.$    
\end{proof}

\begin{rem}  Theorem \ref{WCSvan} highlights the difference between finite dimensional CS classes on $M$ and WCS classes on $LM$.
Let dim$(M)=3.$ 
The only invariant monomials of degree two involving the ordinary trace are $\tr(A_1A_2)$ and 
$\tr(A_1)\tr(A_2)$ (corresponding to $ch_{[2]}$ and $c_1^2$, respectively).  

$\tr(A_1A_2)$ gives rise 
to the classical Chern-Simons invariant for $M$. The Chern-Simons class associated to 
$\tr(A_1)\tr(A_2)$ is trivial: this form involves $\tr(\omega_1-\omega_0)\tr(\Omega_t)$, 
which vanishes since  $\omega_1-\omega_0, \Omega_t$ take values in skew-symmetric endomorphisms.

On $LM$, we know that the WCS class $CS^W_3$ associated to 
$\tr(A_1A_2)$ vanishes.  The WCS form associated to $\tr(A_1)\tr(A_2)$ involves 
$\tr\sigma_{-1}(\omega_1-\omega_0) = \tr\sigma_{-1}(\omega_1)$ and $\tr\sigma_{-1}(\Omega_t).$  
Both $\omega_1$ and $ \Omega_t$ take values in skew-symmetric $\pdo$s, but 
it does not follow that these  lower order terms in their symbol expansions are skew-symmetric.  In fact, a calculation using \cite[Lemma A.1]{MRT3} shows that $\sigma_{-1}(\omega_1)$ is not skew-symmetric. 
Thus the WCS form associated to $\tr(A_1)\tr(A_2)$ may be nonzero.  
However, by Remark \ref{nonclosed}(i),
this WCS form may not be closed.  


\end{rem}

\section{{\bf 
WCS Classes and Diffeomorphism Groups of Sasakian $5$-Manifolds }}
\label{dimfive}

In this section we produce several infinite families of nonhomeomorphic 
$5$-manifolds $\overline M$ with
$\pi_1(\diff(\overline M)) =\pi_1(\diff(\overline M), \tid)$ infinite.  These manifolds are the total space of circle bundles over integral 
K\"ahler surfaces. To give some context, although there are many results about diffeomorphism groups of manifolds of dimension less than four and a few results in dimension four, the best results in higher dimensions
work only in the stable range.
 In particular, $\pi_1(\diff(S^5))$ is not well understood; we will see that this is a surprisingly difficult case from our viewpoint.  

The fundamental observation is that a circle action $a: S^1\times \overline M\to\overline M$ on a closed, oriented  manifold 
$\overline M$ can be thought of both as a map $a^D: S^1\to \diff(\overline M)$ and as a map $a_L: M\to {\rm Maps}(S^1, \overline M) = L\overline M.$
The first interpretation gives an element $[a^D]\in \pi_1(\diff(\overline M))$, and the second gives an element
$[a^L] = a^L_*[\overline M]\in H_*(L\overline M).$  As we will prove, the nontriviality of $[a^D]$ is guaranteed if
$\int_{[a^L]} CS^W_5\neq 0$.

The $5$-manifolds we consider come with circle actions given by rotating the fibers. 
The calculation of $\int_{[a^L]} CS^W_5$ is feasible only for special metrics.
Fortunately, circle bundles over K\"ahler manifolds come with Sasakian structures, and in particular with specific metrics related to the K\"ahler form.  For these metrics, we can estimate the integral and often show it is nonzero.  The main result (Thm.~\ref{bigthm}) is that for an integral K\"ahler surface $(M,\omega)$, $\pi_1(\diff(\overline M_k))$ is infinite for $k\in \Z, |k|\gg 0$, where $\overline M_k$ is the circle bundle associated to $k\omega.$  

In \S3.1 we discuss the relationship between $\pi_1(\diff(\overline M))$ and $L\overline M$.  The main result is proven is \S3.2, and examples are given in \S3.3.
\medskip

\noindent {\bf Notation:} In this section, $\pi_1(\diff(N)) = \pi_1(\diff(N),\tid)$ for a manifold $N$.


\subsection{Circle actions}

Recall that  $H^*(LM)$ denotes de Rham cohomology of complex valued 
forms.   In particular, integration of closed forms over homology cycles gives a pairing of
$H^*(LM)$ and $H_*(LM,\BbC)$. 

 For a closed, oriented manifold $M$, let $a_0,a_1:S^1\times M\to M$ be two smooth actions.  

\begin{defn}\label{defsmooth} (i)  $a_0$ and $a_1$
 are {\it smoothly  homotopic} if there exists a smooth map 
$$F:[0,1]\times S^1\times M\to M,\ F(0,\theta,m) = a_0(\theta,m),\
 F(1,\theta,m) = a_1(\theta,m).$$

(ii)  $a_0$ and $a_1$ are {\it smoothly  homotopic through actions} if
 $F(t,\cdot,\cdot):S^1\times M\to M$ is an action for all $t$.

\end{defn}

An action $a$ can be rewritten in two equivalent ways.

\begin{itemize}
\item $a$ determines (and is determined by) 
$a^D:S^1\to \diff(M)$ given by
$a^D(\theta)(m) = a(\theta,m).$  
Since $a^D(0) = {\rm Id},$  we get a class $[a^D]\in \pi_1(\diff(M), {\rm Id})$.
Here Diff$(M)$ is a Hilbert manifold
as an open subset of the Hilbert manifold of $\maps(M) = \maps(M,M)$ of some fixed high Sobolev class; one can also work with  $\diffm$ as a Fr\'echet manifold, since the homotopy type 
of $\diffm$ is independent of the topology by Appendix B.  In either topology, it is 
straightforward to check that $a$ is smooth iff $a^D$ is smooth.

\item $a$ determines (and is determined by)
$a^L:M\to LM$ given by $ a^L(m)(\theta) = a(\theta,m)$. This determines a class
$[a^L] = a^L_*[M]\in H_n(LM,\Z)$ with $n = {\rm dim}(M)$.  Again, $a$ is smooth iff $a^L$ is smooth in either topology.

\end{itemize}

In general, we have a sequence of maps
\begin{eqnarray*}\pi_1(\Diff(M)) &=& \maps(S^1, \Diff(M))/{\sim} 
\stackrel{\alpha}{\to} \maps(S^1\times M, M)/{\sim}\\
&=& \maps (M, LM)/{\sim} \stackrel{\beta}{\to} H_n(LM,\Z),
\end{eqnarray*}
where $\sim$ indicates modulo homotopies, taken to be smooth in the sense of Def.~\ref{defsmooth}.
Here $\alpha(f)(\theta, m) = f(\theta)(m)$, and
$\beta(g) = g_*[M].$  
$$ \alpha \mapsto \alpha' \mapsto \alpha'_*[M].$$

Set $\beta\alpha = \eta.$

\begin{prop} \label{prop:insert}$\eta:\pi_1(\diff(M))\to H_n(LM,\Z)$ is a homomorphism, and 
$\eta([a^D]) = [a^L].$
\end{prop}

\begin{proof}  
We claim that $a_0$ is smoothly homotopic to $a_1$ iff $a^L_0,
a^L_1:M\to LM$ are smoothly homotopic. In one direction, 
let $F$ be the homotopy from $a_0$ to $a_1$.  Set
	  $H:[0,1]\times M \to LM $ by $H(t,m)(\theta) = F(t,\theta,m).$  Then
$H(0,m)(\theta) = F(0,\theta,m) = a_0(\theta,m) =  a^L_0(m)(\theta)$,
  $H(1,m)(\theta) =  a^L_1(m)(\theta),$  so $H$ is a homotopy from $
  a^L_0$ to $ a^L_1.$  
It is  easy to check that $H$ is smooth.

Conversely, if $H:[0,1]\times M \to LM $ is a smooth homotopy from $a^L_0$ to
$a^L_1$, set $F(t,\theta, m) = H(t,m)(\theta).$  

From the claim, it follows that 
if $a_0$ is smoothly homotopic to $a_1$, then
$[a^L_0] =  [a^L_1]\in H_n(LM,\Z).$  For
$a^L_0$ and $a^L_1$ are homotopic, so
$[a^L_0] = a^L_{0,*}[M] = a^L_{1,*}[M] = [a^L_1].$

This shows that $\eta$ is well defined.
As in the setup of the Hurewicz theorem, it is elementary that $\eta$ is a homomorphism.  
\end{proof}

\begin{cor}\label{cor:one} If $a_0$ is smoothly homotopic to $a_1$, then
$[a^L_0] =  [a^L_1]\in H_n(LM,\Z).$  
\end{cor}

\begin{proof} By the previous proof, $a^L_0$ and $a^L_1$ are homotopic. Thus 
$[a^L_0] = a^L_{0,*}[M] = a^L_{1,*}[M] = [a^L_1].$
\end{proof}

The subset of $\pi_1(\diff(M))$ given by homotopy classes of smooth actions is not a subgroup, but we do have the following:

\begin{lem}\label{lem:one}  (i) $a_0$ is smoothly homotopic to $a_1$ through actions implies $[a^D_0] =
	  [a^D_1]\in \pi_1(\diff(M)).$
	  
(ii) $[a^D_0] =
	  [a^D_1]\in \pi_1(\diff(M))$ implies $a_0$ is smoothly homotopic to $a_1.$   
\end{lem}

\begin{proof} (i) Given $F$ as above, set $G:[0,1]\times S^1\to
	  \diff(M)$ by $G(t,\theta)(m) = F(t,\theta,m).$  We have $G(0,\theta)(m)
	  = a_0(\theta,m) = a^D(\theta)(m)$, $G(1,\theta)(m) = a_1(\theta, m)
	  = a^D_1(\theta)(m)$.
$G(t,\theta)\in\diff(M)$, as
$$G(t,-\theta)(G(t,\theta)(m)) = F(t,-\theta,F(t,\theta,m)) = F(t,0,m) = m,$$
since $F(t,\cdot,\cdot)$ is an action.
Since $F$ is
smooth, $G$ is a continuous (in fact smooth) map of $\diff(M)$.
Thus $a^D_0, a^D_1$ 
are homotopic as elements of 
 $\maps(S^1,\diff(M))$, so $[a^D_0] = [a^D_1].$
\medskip

\noindent (ii)  Let $G:[0,1]\times S^1\to \diff(M)$ be a continuous
homotopy from
$a^D_0(\theta) = G(0,\theta) $ to $a^D_1(\theta) = G(1,\theta)$ with $G(t,1) = {\rm Id}$
for all $t$. 
$G$ gives a continuous map $G_1:[0,1]\times S^1\times M\to M$, which by 
\cite[Thm. 10.21]{lee} can be approximated arbitrarily well by a smooth map with the same boundary behavior as $G_1.$  This gives an arbitrarily good approximation of
$G$ by a smooth map, also called $G$. Set 
$F:[0,1]\times S^1\times M\to M$ by
$F(t,\theta,m) = G(t,\theta)(m).$  
 $F$ is smooth.  Then
$F(0,\theta,m) = G(0,\theta)(m) = a^D_0(\theta)(m) = a_0(\theta,m)$, and
$F(1,\theta,m) = a_1(\theta,m).$  Thus $a_0$ and $a_1$ are smoothly homotopic.

\end{proof}


This yields a technique to use WCS classes to distinguish actions and to investigate 
$\pi_1(\diffm).$  From now on, ``homotopic" means ``smoothly homotopic."

\begin{prop} \label{prop:two} Let dim$(M)=2k-1.$ Let $a_0, a_1:S^1\times M\to M$ be actions.


(i)  If $\int_{[a^L_0]} CS^W_{\kk} \neq \int_{[a^L_1]} CS^W_{\kk}$, then $a_0$ and $a_1$
  are not homotopic through actions, and $[a^D_0]\neq [a^D_1]\in \pi_1(\diff(M)).$

(ii)  If there exists an action $a$ with $\int_{[a^L]} CS^W_{\kk} \neq 0,$ then
  $\pi_1(\diff(M))$  is infinite.

\end{prop}

\begin{proof} 
(i)  If $[a^L_0] =  [a^L_1]\in 
H_n(LM, \C)$, then  $a^L_0 = a^L_1 + \partial \alpha$ for some $2k$-chain $\alpha.$  
Consider $CS^W_{2k-1}$ as an element of singular cohomology. For $\langle\ ,\  \rangle$ the pairing of cochains and chains, we have
\begin{eqnarray*} \int_{[a^L_0]} CS^W_{2k-1} &=& \langle CS^W_{2k-1}, a^L_0\rangle 
= \langle CS^W_{2k-1}, a^L_1 + \partial \alpha\rangle\\
& =&  \int_{[a^L_1]} CS^W_{2k-1} +
\langle \delta CS^W_{2k-1}, \alpha\rangle 
= \int_{[a^L_1]} CS^W_{2k-1} + \int_{\alpha} dCS^W_{2k-1}\\
&=&   \int_{[a^L_1]} CS^W_{2k-1}.  
\end{eqnarray*} 
Thus $[a^L_0] \neq  [a^L_1]$, so $a_0$ is not homotopic to $a_1$ (Cor.~\ref{cor:one}) and hence not homotopic through actions.  By 
Prop.~\ref{prop:insert}, $[a^D_0] \neq  [a^D_1]$.  

(ii) Let $a_n$ be the $n^{\rm th}$ iterate of 
 $a$, i.e. $a_n(\theta,m) =
a(n\theta,m).$  
We claim that 
 $\int_{[a^L_n]}CS^W_{\kk} =
n\int_{[a^L_1]}CS^W_{\kk}$.  By (\ref{5.4}), every term in $CS^W_{\kk}$ is of the
form $\ints\dot\gamma(\theta) f(\theta)$, where $f$ is a periodic function on the
circle.  Each loop $\gamma\in
a^L_1(M)$ corresponds to the loop $\gamma(n\cdot)\in a^L_n(M).$  Therefore the term
$\ints\dot\gamma(\theta) f(\theta)$ is replaced by 
$$\ints \frac{d}{d\theta}\gamma(n\theta) f(n\theta)d\theta 
 = n\int_0^{2\pi} \dot\gamma(\theta)f(\theta)d\theta.$$
Thus $\int_{[a^L_n]}CS^W_{\kk} = n\int_{[a^L]}CS^W_{\kk}.$
 By (i), the $[a^L_n]\in 
\pi_1(\diff(M))$
are all distinct.  

\end{proof}

\begin{rem}
If two actions
are homotopic through actions,
the $S^1$-index of an equivariant differential operator of the two actions is the same. (Here equivariance
means for each action $a_t, t\in [0,1].$)
In contrast to Proposition \ref{prop:two}(ii), the $S^1$-index of an equivariant operator
cannot distinguish actions on odd dimensional manifolds, as the
$S^1$-index vanishes. This can be seen from the
local version of the
$S^1$-index theorem \cite[Thm. 6.16]{BGV}. For the normal bundle to the
fixed point set is always even dimensional, so the fixed point set consists of
odd dimensional submanifolds.  The integrand in the fixed point submanifold
contribution to the $S^1$-index is the constant term in the short time
asymptotics of the appropriate heat kernel.  In odd dimensions, this constant
term is zero.
\end{rem}

\subsection{Sasakian structures over K\"ahler surfaces and diffeomorphism groups}

Let $(M^4, g, J,\omega)$ be a compact integral K\"ahler surface, i.e.~$J$ is the complex structure, $g$ is the 
K\"ahler metric, and with K\"ahler form $\omega\in H^2(M,\Z).$  Recall that 
$\omega(X,Y) = g(JX,Y) = -g(X,JY).$  Recall that $M$ is integral iff it is projective algebraic.

Fix $p\in \Z.$ 
As in geometric quantization, we can construct a $S^1$-bundle $L_p\to M$ with connection 
$\be$ with curvature $d\be = p\omega.$  Let $\bmk$ be the total space of $L_p$. Our goal is to show that $\pi_1(\diff(\overline M_p))$ is infinite for $|p|\gg 0.$

$\bmk$ has a Sasakian structure, as we now sketch; see \cite[\S4.5]{blair}, \cite[Lemma 1]{oneill} for details. The horizontal space of the connection is $\calH = {\rm Ker}(\be)$.  We choose a normalized vertical vector
$\bxi$ satisfying 
$$\be(\bxi) = 1.$$
Note that $d\be(\bxi, \cdot) = 0,$ since the curvature form is horizontal.
 In the language of contact geometry, $\xi$ is the characteristic vector field of the fibration $\bmk\stackrel{\pi}{\to} M.$  Define a metric $\bg$ on $\bmk$ by
$$\bg(\bx,\by) = g(\pi_*\bx, \pi_*\by) + \be(\bx)\be(\by).$$
Then $\bxi\perp \calH, \bg(\bxi, \bxi) = 1, \bg(\bxi, \bx) = \be(\bx).$  Moreover, $\bxi$ is a Killing vector field,
since
\begin{eqnarray*} \calL_{\bxi}\bg &=& \calL_{\bxi}\pi^* g + \calL_{\bxi}\be\otimes \be+ \be\otimes \calL_{\bxi}\be\\
&=&  0 + [d\be(\bxi,\cdot) + d(\be(\bxi))] \otimes \be + \be\otimes  [d\be(\bxi,\cdot) + d(\be(\bxi))] \\
&=& 0.  
\end{eqnarray*}
Thus the flow lines of $\xi$ are geodesics, making $\pi:\bmk\to M$ a Riemannian submersion with totally geodesic fiber.

Let $\bp$ be the $(1,1)$-tensor on $\bmk$ defined by
$$\bp(\bx_{\bar p}) = (J[\pi_*\bx]_{\pi(p)})^L,$$
with $\xl$ the horizontal lift of $X\in T_{\pi(p)}M.$  It is easy to check that 
$$\bp(\bxi) = 0, \bp(\xl) = (JX)^L, \bp^2 = -I + \be\otimes\bxi.$$

Let $\bx = \bx^H + \bx^V\in T\bmk$ be the decomposition of $\bx$ into horizontal and vertical components for the Levi-Civita connection $\bn$ associated to $\bg.$  Define $A:TM\to TM$, $H:TM\otimes TM \to \R$ by
$AX = A(X) = \pi_*(\bn_{\xl}\bxi),$ $(\bn_{\xl}\yl)^V = H(X,Y)\bxi.$  Using $(\bn_{\xl}\yl)^L = (\nabla_XY)^L$
\cite[Lem. 1]{oneill} and $\bg( \bn_{\xl}\bxi,\bxi) = 2\xl(\bg(\bxi,\bxi) )= 0$, we get
$$\bn_{\xl} \yl = (\nabla_XY)^L + H(X,Y)\bxi,\  \bn_{\xl}\bxi = (AX)^L.$$
Also,
$$\bg( \yl, \bxi) = 0 \Rightarrow \bg(\bn_{\xl}\yl,\bxi) +
\bg( \bn_{\xl}\bxi, \yl) = 0\Rightarrow H(X,Y) =- g( AX,Y).$$

\begin{lem}  $AX = pJX.$  Equivalently, $H(X,Y) = -pg (JX, Y).$
\end{lem}

\begin{proof}  We compute
\begin{eqnarray*} (\bn_{\bx}\be)(\by) &=&  \bn_{\bx}(\be(\by))
-\be(\bn_{\bx}\by) = \bn_{\bx}(\bg(\bxi, \by)) - \be(\bn_{\bx}\by) \\
&=& \bg(\bn_{\bx}\bxi, \by) + \bg(\bxi, \bn_{\bx}\by)- \be(\bn_{\bx}\by) 
= \bg(\bn_{\bx}\bxi, \by).
\end{eqnarray*}
Thus
\begin{eqnarray*} d\be(\bx, \by) &=& \frac{1}{2}[ (\bn_{\bx}\be)(\by) -  (\bn_{\by}\be)(\bx)]\\
&=& \frac{1}{2}[ \bg(\bn_{\bx}\bxi, \by) - \bg(\bn_{\by}\bxi, \bx)]
= \bg(\bn_{\bx}\bxi, \by),
\end{eqnarray*}
by \cite[Lemma 2]{oneill}.  Thus
$$pg(JX, Y) = d\be(X^L, Y^L) = \bg(\bn_{\xl}\bxi, \yl) = \bg( (AX)^L, Y^L) = g(AX,Y),$$
so 
$$AX = pJX,\  H(X,Y) = -pg(JX,Y).$$

\end{proof}

In summary, we have
\begin{equation}\label{sub1}\bn_{\xl}\yl = (\nabla_XY)^L -pg(JX,Y)\bxi,\ \bn_{\xl}\bxi = \bn_{\bxi}X^L
 = p(JX)^L,\ \bn_{\bxi}\bxi = 0,
\end{equation}
which implies
\begin{equation}\label{sub2} \bn_{\bx}\bxi = p\bar\Phi \bx
\end{equation}
Moreover,
\begin{eqnarray}\label{sub3}\bg(\bar\Phi \bx, \bar \Phi \by) &=& \bg(\bx,\by) - \be(\bx)\be(\by),
\\
(\bn_{\bx}\bar\Phi)( \by) &=& -p(\bg(\bx,\by)\bxi - \be(\by)\bx),\nonumber
\end{eqnarray}
which can be checked case by case for $\bx, \by$ vertical or horizontal.  By definition, the data
$(\bmk, \bp, \bxi, \be, \bg)$ satisfying (\ref{sub2}) and the second equation in (\ref{sub3}) defines a 
Sasakian structure \cite[\S6.3]{blair}.

Following \cite{oneill}, we now compute the curvature  $\overline R$ of $\bg$ in terms of the curvature $R$ of $M$,  although our computations are simpler since the fiber of $\bmk$ is one-dimensional. 

\begin{prop}  \label{3.7}
\begin{eqnarray*}\bg(\br(X^L, Y^L)\zl,W^L) &=& \la R(X,Y)Z,W\ra + p^2[-\la JY,Z\ra\la JX,W\ra\\
&&\quad 
+\la JX,Z\ra\la JY,W\ra +2\la JX,Y\ra \la JZ,W\ra],\\
\bg(\br(\xl, \yl)\zl, \bxi) &=& 0,\\        
\bg(\br(\bxi,\xl)\yl,\bxi) &=& p^2\la X,Y\ra.
\end{eqnarray*}
\end{prop}

\begin{proof}  
Denoting $g(X,Y)$ by 
$\la X,Y\ra$, we have
\begin{eqnarray*} \bn_{\xl}(\bn_{\yl} \zl )&=&\bn_{\xl}( (\nabla_YZ)^L -p\la JY,Z\ra\bxi)\\
&=&[ (\nabla_X\nabla_YZ)^L -p\la JX,\nabla_YZ\ra\bxi ]- p\xl\la JY,Z\ra\bxi 
-p\la JY,Z\ra\bn_{\xl}\bxi\\
&=&(\nabla_X\nabla_YZ)^L  
- p\la JX,\nabla_YZ\ra\bxi 
-p(\la J\nabla_XY,Z\ra +\la JY, \nabla_XZ\ra)\bxi\\
&&\quad - p^2\la JY,Z\ra (JX)^L.
\end{eqnarray*}
Also,
\begin{eqnarray*} [\xl,\yl] &=& \bn_{\xl}\yl-\bn_{\yl}\xl = (\nabla_{X}Y)^L -p\la JX,Y\ra\bxi
- (\nabla_{Y}X)^L +p\la JY,X\ra\bxi\\
&=& [X,Y]^L - 2p\la JX,Y\ra \bxi,
\end{eqnarray*}
so
\begin{eqnarray*} \bn_{[\xl,\yl]}\zl &=& \bn_{[X,Y]^L}\zl - 2p\la JX,Y\ra \bn_{\bxi}\zl\\
&=& (\nabla_{[X,Y]}Z)^L -p\la J[X,Y], Z\ra\bxi -2p^2\la JX,Y\ra (JZ)^L.
\end{eqnarray*}
Thus
\begin{eqnarray*} \br(\xl,\yl)\zl &=& \bn_{\xl}\bn_{\yl}\zl - \bn_{\yl}\bn_{\xl}\zl - \bn_{[\xl,\yl]}\zl\\
&=& (\nabla_X\nabla_YZ)^L- p^2\la JY,Z\ra (JX)^L
-  (\nabla_Y\nabla_XZ)^L+ p^2\la JX,Z\ra (JY)^L\\
&&\quad  - (\nabla_{[X,Y]}Z)^L +2p^2\la JX,Y\ra (JZ)^L\\
&&\quad -p\left[\la JX,\nabla_YZ\ra
+ \la J\nabla_XY,Z\ra + \la JY,\nabla_XZ\ra\right.\\
&&\quad \left.
-\la JY,\nabla_XZ\ra -  \la J\nabla_YX,Z\ra - \la JX,\nabla_YZ\ra
-  \la J[X,Y],Z\ra\right]\bxi\\
&=& (R(X,Y)Z)^L +p^2[-\la JY,Z\ra (JX)^L + \la JX,Z\ra (JY)^L
+2\la JX,Y\ra (JZ)^L].
\end{eqnarray*}
This implies the first two statements of the Proposition.

 For the last statement,
we have
\begin{eqnarray*} \bn_{\xl}(\bn_{\bxi}\yl) &=& \bn_{\xl} (\bn_{\yl}\bxi) = -p\bn_{\xl}J\yl = -p[(J\nabla_XY)^L +p\la JX,JY\ra \bxi],\\
\bn_{[\bxi,\xl]}\yl &=& 0,\\
\bn_{\bxi}\bn_{\xl}\yl &=& \bn_{\bxi}((\nabla_XY)^L + p\la JX,Y\ra \bxi) = \bn_{\bxi}(\nabla_XY)^L\\
&=& \bn_{(\nabla_XY)^L}\bxi = -p(J\nabla_XY)^L,
\end{eqnarray*}
using $\nabla J = J\nabla, \bxi\la JX,Y\ra = 0, \nabla_{\bxi}\bxi = 0.$  Thus
$$\br(\bxi,\xl)\yl = p^2\la X,Y\ra \bxi.$$
\end{proof}

We now compute $CS^W_5$ on $\bmk$.  Set $\dg = \bxi$ and let 
$\{\eb_i\}$ be an orthonormal frame of $T\bmk$ with $\bxi = \eb_1$.  Since $\eb_i$ is perpendicular to $\eb_1$ for $i\neq 1$, we can write  $\eb_i = e^L_i$ for an orthonormal frame 
 $\{e_i\}_{i=2}^5$ of $TM.$ 

Let $a:S^1\times \bmk\to \bmk$ be given by the rotation action in the $S^1$ fibers. As in \S3.1, we want to compute
\begin{equation}\label{want} \int_{[a^L]}CS^W_5 = \int_{a_*[\bmk]}CS^W_5 = \int_{\bmk} a^*CS^W_5.
\end{equation}
By (\ref{csg})
\begin{eqnarray}\label{exp}\lefteqn{a^*CS^W_{5,\gamma}(\eb_1,\ldots,\eb_5) }\nonumber \\
&=&  \frac{3}{5}   
\sum_\sigma \sgn(\sigma)  
\br(a_*\ebo,a_*\eb_\ell, a_*\bxi, a_*\eb_n)\br(a_*\ebt,a_*\ebth, a_*\eb_r,a_*\eb_\ell)
\\
&&\quad \qquad\qquad\cdot \br(a_*\ebf,a_*\ebfi,a_*\eb_n,a_*\eb_r).\nonumber
\end{eqnarray}
Because the action is via isometries, we have e.g.
$\br(a_*\ebo,a_*\eb_\ell, a_*\bxi, a_*\eb_n) = \br(\ebo,\eb_\ell,\bxi,\eb_n),$ and since $CS^W_5$ involves the average of each loop of these curvatures, we will omit $a_*$ in what follows.  In particular, we will write 
$\int_{[a^L]}CS_5^W$ just as $\int_{\bmk} CS^W_5.$

 Choose a  local orthonormal frame of $M$ of the form $\{e_2, Je_2, e_3, Je_3\}.$

\begin{lem}\label{lemapp}  
The Chern-Simons form (\ref{exp}) equals
\begin{eqnarray}\label{bsumna}\lefteqn{a^*CS^W_{5,\gamma}(\bxi, e_2, Je_2, e_3, Je_3) }\nonumber\\
&=&  \frac{3p^2}{5}     \left\{ 32\pi^2 p_1(\Omega)(e_2, Je_2, e_3, Je_3)
+ 32p^2[3 R(e_2, Je_2, e_3, Je_3) -R(e_2, e_3, e_2, e_3)   \right.\nonumber\\
&&\quad \left. 
-  R(e_2, Je_3, e_2, Je_3)  + R(e_2, Je_2, e_2, Je_2)+ R(e_3, Je_3, e_3, Je_3)]\right.\\
 &&\quad \left.+ 192p^4 \right\},\nonumber
\end{eqnarray}
where $p_1(\Omega)$ is the first Pontrjagin form of $(M,g).$   
\end{lem}

\noindent The proof is in  Appendix A.

Set 
$$|R|_\infty = \max\{|R(e_i, e_j, e_k, e_\ell)|: i,j,k,\ell\in\{2,3,4,5\}\}.$$ 
Here we use the previous notation for the orthonormal frame $\{e_2, e_3, e_4, e_5\}.$

\begin{prop} \label{39}$\int_{\bmk} CS^W_5 >0$ if
$$ p^2\left(32\pi^2  p_1(\Omega)(e_2, Je_2, e_3, Je_3) -224p^2 |R|_\infty + 192p^4\right) >0$$ 
pointwise on $M$.  Moreover,
$\int_{\bmk} CS^W_5 >0$ if
\begin{equation}\label{pest} p^2\left(96\pi^2 \sigma(M) -224p^2|R|_\infty  \vol(M) + 192p^4 
\cdot  \vol(M)\right) >0.
\end{equation}
Here $\sigma(M)$ is the signature of $M$.  
\end{prop}

\begin{proof}  The first statement follows from (\ref{bsumna}), noting that 
$\int_{\bmk} CS^W_5 = \int_{\bmk} a^*CS^W_5$ in our notation.  
For the second statement, we use $\vol(\bmk) = \lambda \vol(M)$, where the circle
fiber has constant length $\lambda$ \cite[p. 37]{blair}.  Also, 
$p_1(\Omega) = p_1(\Omega)(e_2, Je_2, e_3, Je_3)e_2^*\wedge (Je_2)^* \wedge e_3^*\wedge (Je_3)^*$ in the obvious notation, so $\int_{M_4} p_1(\Omega)(e_2, Je_2, e_3, Je_3)\dvol_M = \int_M p_1(\Omega)
= 3\sigma(M).$  This gives
$$p^2\left(96\pi^2\lambda \sigma(M) -224p^2|R|_\infty  \lambda\vol(M) + 192p^4 
\cdot \lambda \vol(M)\right) >0,$$
from which the second statement follows.
\end{proof}

\begin{cor} \label{bigcor} The loop of diffeomorphisms of $\bmk$ given by rotation in the circle fiber gives an element of infinite order in $\pi_1(\diff(\bmk))$ provided
$$ |R|_\infty < \frac{6}{7}p^2 + \frac{3\pi^2 \sigma(M)}{7\vol(M)p^2}, \  {\rm if}\ p\neq 0.$$
\end{cor}

\begin{proof}  This follows from Prop.~\ref{prop:two}(ii) and (\ref{pest}).
\end{proof}

\begin{thm}\label{bigthm}  Let $(M^4, J, g, \omega)$ be a compact integral K\"ahler surface, and let $\bmk$ be the circle bundle associated to $p[\omega]\in H^2(M, \Z)$ for $p\in \Z.$  Then
the loop of diffeomorphisms of $\bmk$ given by rotation in the circle fiber gives an element of infinite order in $\pi_1(\diff(\bmk))$  if $|p| \gg 0.$



\end{thm}

\begin{proof}  The estimate (\ref{pest}) holds for $|p| \gg 0.$  Prop.~\ref{prop:two}(ii) then gives the result.
\end{proof}

\begin{rem} 

(i) The Theorem does not apply to the simplest case $p=0$, where
$\overline{M}_0 = M\times S^1.$   In this case, the projection $p:\overline M_0\to S^1$ induces 
$p_*:\pi_1(\diff(\overline M_0))\to \pi_1(S^1) = \Z$. The fiber rotation maps to a generator of $\pi_1(S^1)$, so we immediately conclude $|\pi_1(\diff(\overline M_0))| = \infty.$


(ii) It is not the case that we can scale the metric on $M$ so that the estimate in Cor.~\ref{bigcor} holds for any fixed $p\neq 0.$ 
For if we scale the metric $g$ on $M$ to $\eta g$, $\eta\in \Z^+$, then $|R|_\infty$ scales to $\eta^{-1}
|R|_\infty$, $\vol(M)$ scales to $\eta^2 \vol(M)$, and $\bmk$ changes to $\overline{M}_{\eta p}.$  Thus for fixed $p$, the estimate in Cor. \ref {bigcor} holds for $\eta \gg 0$, which is just a restatement of 
Thm.~\ref{bigthm}.

(iii) By Prop.~\ref{prop:two}(i), these fiber rotations give examples of actions which are not smoothly homotopic to the trivial action.

\end{rem}

\begin{cor}  Under the hypothesis in Thm. \ref{bigthm}, $\pi_1({\rm Isom}(\bmk))$ is infinite.
\end{cor}

\begin{proof} The loop $\beta$ of fiber rotations consists of isometries.  Under the inclusion $i:{\rm Isom}(\bmk)\to \diff(\bmk)$, if 
$\beta\in\pi_1({\rm Isom}(\bmk))$ has finite order,  then $i_*\beta\in\pi_1(\diff(\bmk))$ has finite order.  This contradicts the hypothesis.
\end{proof}


\subsection{Examples}

For specific K\"ahler surfaces we can precisely determine for which $p$ we have $|\pi_1(\diff(\bmk))| = \infty.$ 

\begin{exm}   $M = T^4$.  For  the flat metric, $|R|_\infty = 0$ and $p_1(\Omega) = 0.$  For the standard 
K\"ahler class, by  (\ref{want}), (\ref{bsumna}), we have
$\int_{\overline{M_p}} CS^W_5  >0$ for $p\neq 0.$  By Prop.~\ref{prop:two}(ii),
 $|\pi_1(\diff(\bmk))| = \infty$ for  $p\neq 0$.  As noted above, this also holds trivially for $p=0.$
\end{exm}

\begin{exm}   $M = \CP^2$.
We use this case to give sharp results for $p$ and to check the constants in Prop.~\ref{39}. 
  The total space  $\overline{\C\mathbb P^2_1}$ of the hyperplane  bundle over $\C\mathbb P^2$ is $S^5$.  As above, the rotational action of the fiber is via isometries and gives an element of 
$\pi_1({\rm Isom}(S^5)) = \pi_1(SO(6)) \simeq \Z_2,$ with generator given by this rotational action.
Thus this element is of order at most
two in $\pi_1(\diff(S^5)).$ 
The line bundle $L_1$ associated to
$\overline{\C\mathbb P^2_1}$ is diffeomorphic to the dual line bundle $L_1^*\simeq L_{-1}$, the
line bundle associated to $\overline{\C\mathbb P^2_{-1}}$,
 by the map $v\mapsto \langle \ ,v\rangle$ for a hermitian metric on $\overline{\C\mathbb P^2_1}$. Thus
  $\overline{\C\mathbb P^2_1}\approx \overline{\C\mathbb P^2_{-1}}$, and the 
 rotational action also has order at most two for $p=-1.$  (In general, $\overline{\CP^2_p} $ is diffeomorphic to $\overline{\CP^2_{-p}}.$)

Thus we must have  $\int_{\overline{ \C\mathbb P^2_{\pm 1}} }CS^W_5=0$, which we verify by computing (\ref{bsumna}) explicitly. Using  the formula for the curvature tensor of $\CP^2$ in \cite[Vol. II, p. 166] {K-N},\footnote{With $c=4$ in their notation, and noting that $R(X,Y,Z,W)$ in \cite{K-N} is the negative of our $R(X,Y,Z,W)$.} a direct computation gives $b_2 = 0, b_3 = b_4 = -192p^2$ in (\ref{b234}) in the Appendix.  
Since $1 = \sigma(\C\mathbb P^2) = \frac{1}{3}\int_{\C\mathbb P^2} p_1$, this gives
\begin{eqnarray} \int_{\overline{\CP^2_1} =  S^5} CS^W_{5}&=&
 \frac{3p^2}{5}   \left( 64\pi^3\int_{\CP^2} p_1(\Omega) -384p^2 {\rm vol}(S^5) 
+ 192p^4 {\rm vol}(S^5) \right)\notag \\
&=&  \frac{3p^2}{5}   
 \left(192\pi^3 -384p^2 \pi^3 + 192p^4 \pi^3\right)\notag \\
&=&\frac{586p^2\pi^3}{5}   (p^2-1)^2, \label{eqn:CP2}
\end{eqnarray}
which vanishes iff $p= \pm 1$.  Note that the WCS form vanishes pointwise at $p=\pm 1.$

In particular, for any K\"ahler form $\omega = \omega^{FS} + \partial\bar\partial\log f$
in the same cohomology class as the Fubini-Study form $\omega^{FS},$ the curvature estimate in Cor. \ref{bigcor}  fails at $p=1.$ This yields a lower bound for $|R|_\infty$.   Since vol$(\CP^2) = \pi^2/2$ independent of $\omega$, we have
$$|R|_\infty \geq 12/7$$
for any such $\omega.$
\begin{prop}  (i) For $p = \pm 1,$ the element of $\pi_1(\diff(\overline{\CP^2_p}))$ given by rotation in the fiber 
of $\overline{\CP^2_p}\to\CP^2$ has order at most two.  For $p\neq \pm 1$, this element has infinite order.

(ii) For $|p|\neq |\ell|$, $\overline{\CP^2_p}$ is not diffeomorphic to $\overline{\CP^2_\ell}$.

(iii) $\overline{\CP^2_p}$ is diffeomorphic to the lens space $\calL_p = S^5/\Z_p$, where $z\sim
e^{2\pi i/p}z$ for $z\in S^5.$
\end{prop}

\begin{proof} Part (i)  follows from  Prop.~\ref{prop:two}(ii), (\ref{want}), and (\ref{eqn:CP2}). 

(iii) implies (ii), but for later purposes we prove (ii) directly. For (ii),  the Euler class of the fibration $S^1\to \overline{\CP^2_p}\to \CP^2$ is $p$ times the generator $[\omega = \omega^{FS}]$ of
$H_2(\CP^2,\Z).$ As in 
\cite[Appendix A]{gdsw}, one piece of  
the Gysin sequence of this fibration is
$$0\to H^3(\bmk)\to \Z \stackrel{\cup p[\omega]}{\to} \Z \to H^4(\bmk)\to 0,$$
which implies
$
 H^4(\overline{\CP^2_p},\Z)\simeq \Z/|p|\Z$ for all $p$.  This gives (ii).
 
 For (iii), let $\phi_i: S^5 |_{U_i}= L_1|_{U_i}\to U_i\times S^1$ be local trivializations of $L_1$.  
 Set $e_p:S^1\to S^1, z\mapsto z^p$.
 The $\Z_p$ action preserves the fibers of $L_1$, so 
 the maps $\phi_i^p = e_p\circ\phi_i:\calL_p|_{U_i}\to U_i\times S^1$, 
 are local trivializations of $\calL_p$ as a circle bundle over $\CP^2.$

The
  first Chern class $c_1(L_1)$ is a  generator of 
 $H^2(\CP^2)$, and has the \v Cech representative
 $(2\pi i )^{-1}(\ln \phi_{ij} + \ln \phi_{jk} + \ln \phi_{ki}),$  where $\phi_{ij} = \phi_i\circ\phi_j^{-1}.$  
 Thus $c_1(\calL_p) = p c_1(L_1)= c_1(\overline{\CP^2_p}).$  Therefore $\calL_p\simeq \overline{\CP^2_p}$ as line bundles, from which it follows that they are diffeomorphic.

\end{proof}
\end{exm}

\begin{exm} \label{2p1} $M = \CP^1\times\CP^1$.  As for $\CP^2$, we can compute explicitly the Chern-Simons class. Let $\omega_1, \omega_2$ be the standard 
K\"ahler form on each $\CP^1$ of sectional curvature $1$.  Let $R$ be the curvature of the K\"ahler metric for the 
K\"ahler form $\omega = a\omega_1+b\omega_2$ for $a,b>0$.
For $\{e_2, Je_2\}$ and
$\{e_3, Je_3\}$ orthonormal frames for the first and second $\CP^1$, we have
\begin{eqnarray*}R(e_2,J e_2, e_3, Je_3) &=& R(e_2, e_3, e_2, e_3) = R(e_2, Je_3, e_2, Je_3) = 0,\\ 
R(e_2, Je_2, e_2, Je_2) &=& -a^{-1},\ R(e_3, Je_3, e_3, Je_3) = -b^{-1}.
\end{eqnarray*}
For  $a,b\in\Z^+$,  $\omega = a\omega_1+b\omega_2$ is an integral K\"ahler form. 
Let $\mabp$ be the total space of the line bundle associated to $p\omega.$ 
Since $\sigma(M) = 0$, the integral of (\ref{bsumna}) becomes
$$  \frac{3p^2}{5}   \int_{\mabp} (-32p^2(-a^{-1}-b^{-1}) + 192p^4)\dvol,$$
which is positive for $p\neq 0.$

This produces a new infinite family of five manifolds with $\pi_1(\diff(\mabp))$ infinite.

\begin{prop} \label{ntl} (i) Let $M = \CP^1\times \CP^1 = S^2\times S^2$. For  $a, b\in \Z^+ ,p \in\Z$, the element of $\pi_1(\diff(\mabp))$ given by rotation in the fiber 
of $\mabp\to M$ has infinite order.

(ii) For ${\rm gcd}(a,b)\neq {\rm gcd}(c,d)$, $\mab$ is not diffeomorphic to $\overline{M}_{(c,d)}.$ Furthermore, $\mab$ is not diffeomorphic to $\overline{\CP^2_p}$ for
any $p$.
\end{prop}

\begin{proof} (i) is proven as in the previous examples.

 (ii) Part of the Gysin sequence with $\Z$ coefficients is
$$ H^3(M)\to H^3(\mab)\stackrel{\beta}{\to} H^2(M)\stackrel{\alpha}{\to} H^4(M)\to
H^4\mab)\to H^3(M),$$
where $\alpha = \cup (a[\omega_1] + b[\omega_2])$. This reduces to 
$$0\to H^3(\mab)\stackrel{\beta}{\to}\Z\oplus \Z \stackrel{\alpha}{\to} \Z \to H^4(\mab) \to 0.$$
Since $\alpha(x,y) = ax+by$, the image of $\alpha$ is gcd$(a,b)\Z$.  Thus $H^4(\mab)\simeq \Z_{{\rm gcd}(a,b)}$, which gives the first statement in (ii).  

Since Im$(\beta)\simeq {\rm Ker}(\alpha)\simeq\{(bx,-ax):x\in\Z\}\simeq \Z$, we have $H^3(\mab)\simeq \Z.$  The corresponding part of the Gysin sequence for $\overline{CP^2_p}$ yields $H^3(\overline{CP_p^2}) \simeq 0$, which gives the second statement in (ii).
\end{proof}

\end{exm}

\begin{exm}  Let  $M$ be a compact projective K\"ahler surface with a Ricci flat metric, i.e. a compact projective K3 surface.
We orient $M$ with the opposite of the orientation coming from the complex structure. Let
$R:\Lambda^2(M)\to\Lambda^2(M)$ be the curvature operator. According to the decomposition $\Lambda^2(M) = \lp \oplus \lm$ into $\pm 1$ eigenspaces of the star operator, we have

\centerline{\hskip .02 in $\lp$\hskip .25in $\lm$}
\vskip -.1 in
$$R =  \left(\begin{array}{cc}\rpp& R_{+-}\\ R_{-+}&\rmm\\
\end{array}\right)
\begin{array}{c} {\lp}\\
\lm\end{array}$$
\medskip

\noindent   $\Tr(\rpp) + \Tr(\rmm)= 4r = 0$, where $r$ is the scalar curvature. 
 $R_{+-}$ and $R_{-+}$ are determined by the traceless Ricci tensor, and so 
vanish.  

Let $(e_2, Je_2, e_3, Je_3)$ be a local orthonormal frame for $TM$ with $*e_2 = -Je_2\wedge e_3\wedge Je_3$, etc.; the minus sign reflects the change in orientation.  Then a basis of $\lp$ is 
$$\{f_1^+ = e_2\wedge Je_2 - e_3\wedge Je_3, f_2^+ = e_2\wedge e_3 + Je_2\wedge Je_3, 
f_3^+ = e_2\wedge Je_3 - Je_2\wedge e_3\},$$
and a basis of $\lm$ is 
$$\{f_1^- = e_2\wedge Je_2 + e_3\wedge Je_3, f_2^- = e_2\wedge e_3 - Je_2\wedge Je_3,
 f_3 ^-= e_2\wedge Je_3 + Je_2\wedge e_3\}.$$
Since $\rpp, \rmm$ are symmetric, there are bases $\{\omega^+_1, \omega^+_2,\omega^+_3\}$ of $\lp$ 
and $\{\omega^-_1, \omega^-_2,\omega^-_3\}$ of $\lm$ for which
$\rpp,\rmm$ diagonalize:
\begin{eqnarray*}\rpp &=& \lambda^+_1\omega^1_+\otimes \omega^+_1 + \lambda^+_2\omega_+^2\otimes \omega^+_2
+  \lambda^+_3\omega_+^3\otimes \omega^+_3\\
\rmm &=& \lambda^-_1\omega^1_-\otimes \omega^-_1 + \lambda^-_2\omega_-^2\otimes \omega^-_2
+  \lambda^-_3\omega_-^3\otimes \omega^-_3,
\end{eqnarray*}
for $\omega_{\pm}^i = (\omega_i^{\pm})^*.$
Thus 
\begin{equation}\label{lsum} \sum_{i=1}^3 \lambda_i^+ +  \sum_{i=1}^3 \lambda_i^-=0.
\end{equation}

We claim that the curvature terms in (\ref{bsumna}) add to zero. By the Bianchi identity, the first curvature term is
\begin{eqnarray*} 3R(e_2,Je_2,e_3,Je_3) &=& -3R(e_2,e_3,Je_3,Je_2) -3R(e_2,Je_3, Je_2, e_3)\\
&=& 3R(e_2,e_3,e_2,e_3) + 3R(e_2,Je_3, e_2, Je_3).
\end{eqnarray*}
Thus the curvature terms in (\ref{bsumna}) become
\begin{equation}\label{rew}
2R(e_2,e_3,e_2,e_3) + 2 R(e_2,Je_3, e_2, Je_3) + R(e_2, Je_2, e_2, Je_2) + R(e_3, Je_3, e_3,
Je_3).
\end{equation}
To compute these terms, we note that
e.g. $R(e_2, Je_2, e_3, Je_3) = -\langle R(e_2\wedge Je_2), e_3\wedge Je_3\rangle.$ 
Write $ f^+_j = b_j^i\omega^+_i, f^-_j = d_j^i\omega_i^-.$
After normalizing $\{\omega^\pm_i\}, \{f_j^\pm\}$, we may assume that $(b_j^i), (d_j^i)$ are orthonormal. 
Then
\begin{eqnarray*} 
R(e_2, e_3, e_2, e_3) &=& -\langle R(e_2\wedge e_3), e_2\wedge e_3\rangle\\
& =&
-\left \langle \rpp\left(\frac{f_2^++ f_2^-}{2}\right), \frac{f_2^+ + f_2^-}{2}\right\rangle
-\left \langle \rmm\left(\frac{f_2^++ f_2^-}{2}\right), \frac{f_2^+ + f_2^-}{2}\right\rangle\\
&=& - \frac{1}{4} \langle \rpp f_2^+, f_2^+\rangle - \frac{1}{4} \langle \rmm f_2^-, f_2^-\rangle\\
&=& -\frac{1}{4}[\lambda_1^+(b^1_2)^2+\lambda_2^+(b_2^2)^2 + \lambda_3^+(b_3^3)^2
+  \lambda_1^-(d^1_2)^2+\lambda_2^-(d_2^2)^2 + \lambda_3^-(d_3^3)^2].
\end{eqnarray*}
Similarly computing the other terms in (\ref{rew}), we get 
\begin{eqnarray*}
\lefteqn{2R(e_2,e_3,e_2,e_3) + 2 R(e_2,Je_3, e_2, Je_3) + R(e_2, Je_2, e_2, Je_2) 
+ R(e_3, Je_3, e_3,Je_3)}\\
&=& -\frac{1}{2}[\lambda_1^+(b^1_2)^2+\lambda_2^+(b_2^2)^* + \lambda_3^+(b_2^3)^2
+  \lambda_1^-(d^1_2)^2+\lambda_2^-(d_2^2)^* + \lambda_3^-(d_2^3)^2]\\
&& -\frac{1}{2}[\lambda_1^+(b^1_3)^2+\lambda_2^+(b_3^2)^* + \lambda_3^+(b_3^3)^2
+  \lambda_1^-(d^1_3)^2+\lambda_2^-(d_3^2)^* + \lambda_3^-(d_3^3)^2]\\
&& -\frac{1}{4}[\lambda_1^+(b^1_1)^2+\lambda_2^+(b_1^2)^* + \lambda_3^+(b_1^3)^2
+  \lambda_1^-(d^1_1)^2+\lambda_2^-(d_1^2)^* + \lambda_3^-(d_1^3)^2]\\
&& -\frac{1}{4}[\lambda_1^+(b^1_1)^2+\lambda_2^+(b_1^2)^* + \lambda_3^+(b_1^3)^2
+  \lambda_1^-(d^1_1)^2+\lambda_2^-(d_1^2)^* + \lambda_3^-(d_1^3)^2]\\
&=& -\frac{1}{2}[\lambda^+_1 + \lambda^+_2 + \lambda^+_2 + \lambda^-_1 + \lambda^-_2
 + \lambda^-_3]\\
&=& 0.
\end{eqnarray*}
Since $p_1(\Omega) = \frac{1}{4\pi^2}(|W_+|^2 - |W_-|^2) = \frac{1}{4\pi^2}|W_+|^2 \geq 0$, 
$\sigma(M) \geq 0.$  (In fact, $\sigma(M) = 16$ in this orientation.) Thus $\int_{\bmk}CS^W_5 \geq 192p^6 >0$ for $p\neq 0.$

\begin{rem}  By \cite{lebrun2}, the
vanishing of scalar curvature on a K\"ahler surface implies the metric is self-dual (after a change of orientation): for the decomposition of the Weyl tensor $W = W_+ + W_-,$ we have $W_-=0$.  Therefore all the pieces in the decomposition of $\rmm$ into projected traceless Ricci, $W_-$ and $r$ vanish,
so $\rmm = 0.$  Thus the calculation above can be somewhat shortened.
\end{rem} 

This result provides another infinite family of examples. To introduce the notation, recall that
$H^2(M;\Z) \simeq \Z^{22}.$ Fix an integral K\"ahler class
$[\omega] = [\omega_1,\ldots,\omega_{22}]$ in the obvious notation and take $a_1,\ldots a_{22}\in \Z^+
\setminus \{0\}.$ For $p\in \Z$, let $\bpa$ be the total space of the line bundle associated to 
$p\sum_{i=1}^{22} a_i\omega_i.$  

\begin{prop}  Let $M$ be a compact projective algebraic K3 surface.

(i) The element of $\pi_1(\diff(\bpa))$ given by rotation in the fiber 
of $\bpa\to M$ has infinite order. 

(ii) For $p=1$, there are infinitely many nondiffeomorphic $\bpa.$  None of these examples are diffeomorphic
 to $\overline{\CP^2_\ell}$ or to the manifolds $\mab$ in Proposition \ref{ntl}.

\end{prop}

\begin{proof}   (i)  This follows from $\int_{\bmk}CS^W_5 >0$ as before.

(ii) 
 As before, $H^4(\ba)\simeq \Z_{{\rm gcd}(a_1,\ldots,a_{22})},$ so infinitely many of the $\ba$ are 
 nondiffeomorphic.    Similarly, $H^4(\bpa)\simeq \Z_{{\rm gcd}(pa_1,\ldots,pa_{22})},$ which easily 
 implies that $H^3(\bpa,\R) \simeq \R^{21}.$  This gives the last statement in (ii).

\end{proof}
\end{exm}

\begin{exm} $S^2\times S^3$.
We now apply a variant of these methods to a Sasaki-Einstein metric on $S^2\times S^3$
constructed in \cite{gdsw} to prove the following:

\begin{prop}  $\pi_1(\diff(S^2\times S^3))$ is infinite.
\end{prop}

\begin{proof}
For $a\in (0,1]$, the  metric 
\begin{eqnarray}\label{metric} g &=&\frac{1-y}{6}(d\theta^2 + \sin^2\theta d\phi^2) + 
\frac{1}{w(y)q(y)} dy^2 + \frac{q(y)}{9}[d\psi^2 -\cos\theta d\phi^2]\nonumber\\
&&\qquad + w(y)\left[d\alpha + \frac{a-2y+y^2}{6(a-y^2)}[d\psi -\cos\theta d\phi]\right]^2,
\end{eqnarray}
with 
$$w(y) = \frac{2(a-y^2)}{1-y}, q(y) = \frac{a-3y^2+2y^3}{a-y^2}, a\in \R,$$
is a family of Sasaki-Einstein metrics on a   $5$-ball in the variables
$(\phi, \theta, \psi, y, \alpha)$  \cite{gdsw}.  
For $p,q$ relatively prime, $q<p$, and satisfying $4p^2-3q^2 = n^2$ for some integer $n$,  
and for  $a = a(p,q)\in (0,1)$, the metric extends to a $5$-manifold $Y^{p,q}$ which has
the coordinate ball as a dense subset.  
In this case, $( \phi, \theta, \psi, y)$
are  spherical coordinates on $S^2\times S^2$ with a nonstandard metric, and $\alpha$ is the fiber coordinate of an 
$S^1$-fibration $S^1\to Y^{p,q}\to S^2\times S^2.$  
$Y^{p,q}$ is diffeomorphic to $S^2\times S^3,$ and
 has first Chern class which integrates over the two $S^2$ factors to $p+q$ and $p$ \cite[\S2]{gdsw}.
The coordinate ranges are $\phi\in (0,2\pi), \theta \in (0,\pi), \psi\in (0,2\pi)$,
$\alpha\in (0,2\pi\ell)$,
where $\ell = \ell(p,q)$, and 
$y\in (y_1, y_2)$, with the $y_i$ the two smaller roots of $a-3y^2+2y^3=0$.  $p$ and
$q$ determine $a, \ell, y_1, y_2$ explicitly \cite[(3.1), (3.4), (3.5), (3.6)]{gdsw}.

For these choices of $p, q$, we get an $S^1$-action $a$ on $Y^{p,q}$ by rotation in the $\alpha$-fiber.
We claim that for e.g. $(p,q) = (7,3)$,
\begin{equation}\label{neqo} \int_{[a^L]} CS_5^W(g) \neq 0.
\end{equation}
By Proposition \ref{prop:two}(ii), this implies $\pi_1(\diff(S^2\times S^3))$ is infinite.

Set $M = S^2\times S^3$. Since $a^L:M\to LM$ has degree one on its image,
\begin{equation}\label{cswf} 
\int_{[a^L]} CS_5^W(g) = \int_M a^{L,*} CS_5^W(g).
\end{equation}
For $m\in M$, 
$$a^{L,*}CS_5^W(g)_m = f(m)d\phi\wedge d\theta\wedge dy\wedge d\psi\wedge 
d\alpha$$ 
for some $f\in C^\infty(M)$.  We determine
$f(m)$ by explicitly computing $a_{*}^L(\partial_\phi),..., a_{*}^L(\partial_\alpha),$ (e.g. 
$a_{*}^L(\partial_\phi)(a^L(m))(t) = \partial_\phi|_{a(m,t)}$ ),
and noting
\begin{eqnarray}\label{f} f(m) &=& f(m)d\phi\wedge d\theta\wedge dy\wedge d\psi\wedge 
d\alpha(\partial_\phi,\partial_\theta,\partial_y,\partial_\psi,\partial_\alpha)\nonumber\\
&=& a^{L,*}CS_5^W(g)_m(\partial_\phi,...,\partial_\alpha)\\
&=& CS_5^W(g)_{a^L(m)}(a_{*}^L(\partial_\phi),...,a_{*}^L(\partial_\alpha)).\nonumber
\end{eqnarray}
Since $CS^W_5(g)$ is explicitly computable from (\ref{csg}), we can
compute $f(m)$ from (\ref{f}).  Then $\int_{[a^L]} CS_5^W(g) = \int_{M} f(m)
d\phi\wedge d\theta\wedge dy\wedge d\psi\wedge 
d\alpha$ can be computed as an ordinary integral in the dense coordinate space.

Via this method, in the Mathematica file {\tt ComputationsChernSimonsS2xS3.pdf} at {\tt http://math.bu.edu/people/sr/},  $\int_{[a^L]} CS_5^W(g)$ is computed as
 a function of $(p,q).$  For example, for
$(p,q) = (7,3)$,
$$\int_{[a^L]} CS_5^W(g)  = -\frac{3}{5}  \cdot \frac{1849\pi^4}{22050} = -\frac{
1849\pi^4}{37750}.$$
This formula is exact; the rationality up to $\pi^4$ follows from $4p^2-3q^2$ being a perfect
square, as then the various integrals computed in (\ref{cswf}) with respect to our coordinates
are rational functions evaluated at rational endpoints. 
 In particular, (\ref{neqo})
holds.
  \end{proof}

\begin{rem}  
For $a=1$, the metric extends to the closure of the coordinate chart, but the total space is $S^5$ with the standard metric.  For $n\gg 0$,
$\pi_1(\diff(S^n))$ is torsion \cite{F-H}. If this  holds below the stable range, specifically for  $n=5$, then
by Proposition \ref{prop:two}(ii), $\int_{[a^L]} CS^W_5 = 0$
for any circle action on $S^5.$  In the formulas in the Mathematica file, $\int_{[a^L]}CS^W_5$ is proportional to
$(-1+a)^2$, which vanishes at $a=1$.  
This gives some plausibility to the conjecture that $\pi_1(\diff(S^5))$ is torsion.
\end{rem}
\end{exm}

\appendix\section{Proof of Lemma \ref{lemapp}}

Recall that we choose a local orthonormal frame of $\overline {M}_p$ 
of the form $\bar e_1 = \xi, \bar e_i = e_i^L$, $i= 2,3,4,5$, where $\{e_i\}$ is a local orthonormal frame of $M$ and $\bar e_i^L$ are horizontal lifts.  Eventually we will refine the frame of $M$ to be of the form $\{e_2, Je_2, e_3, Je_3\}.$

We wish to simplify (\ref{exp}):

\begin{eqnarray}\label{expa}\lefteqn{a^*CS^W_{5,\gamma}(\eb_1,\ldots,\eb_5) }\nonumber \\
&=&   \frac{3}{5}   
\sum_\sigma \sgn(\sigma)   
\br(a_*\ebo,a_*\eb_\ell, a_*\bxi, a_*\eb_n)\br(a_*\ebt,a_*\ebth, a_*\eb_r,a_*\eb_\ell)
\\
&&\quad \qquad\qquad\cdot \br(a_*\ebf,a_*\ebfi,a_*\eb_n,a_*\eb_r),\nonumber
\end{eqnarray}

We divide the permutations into five cases $A_i$ where $\sigma(i) = 1$, i.e., $\eb_{\sigma(i)} = \bxi.$ 
In formulas like (\ref{expa}), we refer to the curvature terms on the right hand side as the first, second or third term.

We claim that the $A_2$ case contributes zero to the integrand in (\ref{expa}).  Ignoring the factor
 $\frac{3}{5}$, we have
\begin{equation}\label{A2} A_2 = \sum_{\sigma, \sigma(2) = 1} \sgn(\sigma)
\br(\ebo,\eb_\ell, \bxi, \eb_n)\br(\bxi,\ebth, \eb_r,\eb_\ell)\br(\ebf,\ebfi,\eb_n,\eb_r).
\end{equation}
By Prop. \ref{3.7}, in the second term  exactly one of $\eb_r,\eb_\ell$ must be $\bxi$ for a nonzero contribution.
   If $\eb_r = \bxi$, then  the first term must have $\eb_n = \bxi$.  The third term is then zero.  
If $\eb_\ell = \bxi$, then $\eb_r = e_k^L$ and the first term forces $\eb_n = e_n^L$.  By the last equation
in Prop. \ref{3.7}, we have
\begin{eqnarray*} A_2 &=& \sum_{\sigma(2) =1}\sgn(\sigma) \br(e_{\sigma(1)}^L, \bxi, \bxi, e_n^L)
\br(\bxi, e_{\sigma(3)}^L,e_k^L, \bxi) \br(e_{\sigma(4)}^L, e_{\sigma(5)}^L, e_n^L, e_k^L)\\
&=&p^2 \sum_{\sigma(2) =1}\sgn(\sigma) \delta_{\sigma(1),n}\delta_{\sigma(3),k} \br(
e_{\sigma(4)}^L, e_{\sigma(5)}^L,e_{n}^L,e_{k}^L)\\
&=& p^2\sum_{\sigma(2) =1}\sgn(\sigma)\br(e_{\sigma(4)}^L, e_{\sigma(5)}^L,e_{\sigma(1)}^L,e_{\sigma(3)}^L)\\
&=& p^2\sum_{\sigma(2) =1}\sum_{i=2}^4\sum_{\sigma(4) = i}\sgn(\sigma)
\br(e_{\sigma(4)}^L, e_{\sigma(5)}^L,e_{\sigma(1)}^L,e_{\sigma(3)}^L)\\
&=&0,
\end{eqnarray*}
since  by the Bianchi identity, for each $i$ we have
\begin{equation*}
 \sum_{\genfrac{}{}{0pt}{}{\sigma(2) =1}{\sigma(4) = i}}
 \sgn(\sigma)\br(e_{\sigma(4)}^L, e_{\sigma(5)}^L,e_{\sigma(1)}^L,e_{\sigma(3)}^L) = 0.
\end{equation*}
 For $A_3$, again exactly 
one of $\eb_r,\eb_\ell$ must be $\bxi.$  If $\eb_r = \bxi$, then in the third term we must have $\eb_n = \bxi$, which forces this term to vanish. 
If $\eb_\ell = \bxi$, as above we use Bianchi to get zero.  
For $A_4$, either $\eb_n$ or $\eb_r$ equals $\bxi.$ However,  when $\eb_n = \bxi$ 
the first term is zero by Prop.~\ref{3.7}, and when $\eb_r=\bxi$ the third term vanishes for the same reason. Thus 
$A_4 = 0.$  The exact same argument shows that $A_5=0.$  

The $A_1 $ term is
$$A_1 = \sum_{\sigma(1) = 1}\sgn(\sigma) \br(\bxi, \eb_\ell, \bxi, \eb_n) \br(\ebt, \ebth, 
\eb_r, \eb_\ell)\br(\ebf,\ebfi, \eb_n, \eb_r).$$
It is easy to check that if any of $\eb_\ell, \eb_n, \eb_r$ equals $\bxi,$ then the contribution to $A_1$ is zero. By Prop. \ref{3.7} (and denoting $(Je_k)^L$ by $Je_k^L$), we have
\begin{eqnarray*}
A_1 &=& \sum_{\sigma(1) = 1}\sgn(\sigma) \br(\bxi, \el_l, \bxi, \el_n)\br(\elt, \elth, \el_r, \el_\ell)
\br(\elf, \elfi, \el_n, \el_r)\\
&=& - p^2 \ssum \delta_{\ell n}\br(\elt, \elth, \el_r, \el_\ell)
\br(\elf, \elfi, \el_n, \el_r)\\
&=& p^2\ssum \br(\elt, \elth, \el_r, \el_\ell)\br(\elf, \elfi, \el_r, \el_\ell)\\
&=& p^2\ssum \br(\elt, \elth, \el_r, \el_\ell)\left[ R(\eef, \eefi, e_r, e_\ell) -p^2\langle J\eef, e_l\rangle
\langle J\eefi, e_r\rangle\right.\\
&&\quad \left.  +p^2\langle J\eef, e_r\rangle\langle J\eefi, e_\ell\rangle + 
2p^2\langle J\eef, \eefi\rangle\langle Je_r, e_\ell\rangle
 \right]\\
 &=& p^2\ssum  \br(\elt, \elth, \el_r, \el_\ell)R(\eef, \eefi, e_r, e_\ell) -p^2 \br(\elt, \elth, J\elfi, J\elf)\\
 &&\quad +  p^2 \br(\elt, \elth, J\elf, J\elfi)
 +2p^2\br(\elt, \elth, e_r^L, J\el_r)\langle J\eef, \eefi\rangle\\
 &=& p^2\ssum \br(\elt, \elth, \el_r, \el_\ell)R(\eef, \eefi, e_r, e_\ell) + 2p^2 \br(\elt, \elth, J\elf, J\elfi)\\
&&\quad  +2p^2\br(\elt, \elth, e_r^L, J\el_r)\langle J\eef, \eefi\rangle
\end{eqnarray*}

Denote the three terms on the right hand side of the last expression by $\aoa, \aob, \aoc.$   Up to a constant, we have
\begin{eqnarray*} \aob &=& \br(\elt, \elth, J\elf, J\elfi)\\
&=& R(\eet, \eeth,J\eef, J\eefi) - p^2\la J\eet, J\eefi\ra \la J\eeth, J\eef\ra\\
&&\quad + p^2\la J\eet, J\eef\ra \la J\eeth, J\eefi\ra
+ 2p^2 \la J\eet, \eeth\ra \la J^2\eef, J\eefi\ra\\
&=& R(\eet, \eeth,J\eef, J\eefi)  -2p^2 \la J\eet, \eeth\ra \la \eef, J\eefi\ra,
\end{eqnarray*}
since $ \la J\eeth, J\eef\ra = \la J\eeth, J\eefi\ra = 0$.
Also,
\begin{eqnarray*} \aoc &=& \br(\elt, \elth, e_r^L, J\el_r)\la J\eef, \eefi\ra\\
&=&R(\eet, \eeth, e_r, Je_r)\langle J\eef, \eefi\ra - p^2\la J\eet, Je_r\ra\la J\eeth, e_r\ra\la J\eef, \eefi\ra\\
&&\quad +p^2 \la J\eet, e_r\ra\la J\eeth, Je_r\ra\la J\eef, \eefi\ra +2p^2
\la J\eet, \eeth\ra \la Je_r, Je_r\ra \la J\eef, \eefi\ra\\
&=& R(\eet, \eeth, e_r, Je_r)\langle J\eef, \eefi\ra- p^2\la \eet, e_r\ra\la J\eeth, e_r\ra\la J\eef, \eefi\ra\\
&&\quad +p^2 \la J\eet, e_r\ra\la \eeth, e_r\ra\la J\eef, \eefi\ra +8p^2
\la J\eet, \eeth\ra  \la J\eef, \eefi\ra,
\end{eqnarray*}
since $\sum_r \la Je_r, Je_r\ra = 4.$  The second and third terms in the last expression are equal, as
$$\la \eet, e_r\ra\la J\eeth, e_r\ra = \la \eet, J\eeth\ra,\ \la J\eet, e_r\ra\la \eeth, e_r\ra
= \la J\eet, \eeth\ra.$$
Thus
\begin{eqnarray*} \aoc &=& R(\eet, \eeth, e_r, Je_r)\langle J\eef, \eefi\ra + 2p^2
\la J\eet, \eeth\ra \la J\eef, \eefi\ra\\
&&\qquad +8p^2 \la J\eet, \eeth\ra  \la J\eef, \eefi\ra\\
&=& R(\eet, \eeth, e_r, Je_r)\langle J\eef, \eefi\ra + 10p^2 \la J\eet, \eeth\ra  \la J\eef, \eefi\ra.
\end{eqnarray*}
Finally,
\begin{eqnarray*}\aoa &=& \br(\elt, \elth, \el_r, \el_\ell)R(\eef, \eefi, e_r, e_\ell) \\
&=& \left[ R(\eet, \eeth, e_r, e_\ell) - p^2\la J\eet, e_\ell\ra \la J\eeth, e_r\ra   +
p^2 \la J\eet, e_r\ra \la J\eeth, e_\ell\ra\right.\\
&&\qquad \left. + 2p^2\la J\eet, \eeth\ra\la Je_r, e_\ell\ra\right] R(\eef, \eefi, e_r, e_\ell)\\
&=& R(\eet, \eeth, e_r, e_\ell)R(\eef, \eefi, e_r, e_\ell) - p^2 R(\eef, \eefi, J\eeth, J\eet)\\
&&\qquad +p^2 R(\eef, \eefi, J\eet, J\eeth) +2p^2 \la J\eet, \eeth\ra R(\eef, \eefi, e_r, Je_r)\\
&=& R(\eet, \eeth, e_r, e_\ell)R(\eef, \eefi, e_r, e_\ell) + 2p^2 R(\eef, \eefi, J\eet, J\eeth)\\
&&\quad +2p^2 \la J\eet, \eeth\ra R(\eef, \eefi, e_r, Je_r).
\end{eqnarray*}

Plugging $\aoa, \aob, \aoc$ into $A_1$, we get
\begin{eqnarray*} A_1 &=& p^2 \ssum \left\{\left[ R(\eet, \eeth, e_r, e_\ell)R(\eef, \eefi, e_r, e_\ell) + 2p^2 R(\eef, \eefi, J\eet, J\eeth)\right.\right.\\
&&\quad\left. \left. +2p^2 \la J\eet, \eeth\ra R(\eef, \eefi, e_r, Je_r)\right]\right.\\
&&\quad +\left. \left[ 2p^2R(\eet, \eeth, J\eef, J\eefi)
-4 p^4\la J\eet, \eeth\ra\la \eef, J\eefi\ra\right]\right.\\
&&\quad +\left.\left[ 2p^2 R(\eet, \eeth, e_r, Je_r)\la J\eef, \eefi\ra 
+20p^4\la J\eet, \eeth\ra  \la J\eef, \eefi\ra\right]\right\}\\
&=& p^2\ssum\left\{ R(\eet, \eeth, e_r, e_\ell)R(\eef, \eefi, e_r, e_\ell) + 4p^2 R(\eef, \eefi, J\eet, J\eeth)\right.\\
&&\quad\left. + 2p^2 \la J\eet, \eeth\ra R(\eef, \eefi, e_r, Je_r)
+2p^2R(\eet, \eeth, e_r, Je_r)\la J\eef, \eefi\ra \right. \\
&&\quad \left.+24p^4\la J\eet, \eeth\ra  \la J\eef, \eefi\ra\right\}.
\end{eqnarray*}

Denote the five terms in the last expression by $b_1,\ldots, b_5.$  Note that
\begin{equation}\label{b1}b_1= -4\tr(\Omega\wedge \Omega)(e_2, Je_2, e_3, Je_3) = 32\pi^2 p_1(\Omega)(e_2, Je_2, e_3, Je_3),
\end{equation} 
where $p_1(\Omega)$ is the first Pontrjagin form.  (As noted before, since the orbits of $\bxi$ have constant length $\lambda$ \cite[p. 37]{blair}, we have
\begin{equation}\label{sig} \int_{\bmk} b_1 \dvol_{\overline M_p}= 32\pi^2 \lambda \int_M p_1(\Omega) = 96\pi^2 \lambda \sigma(M),
\end{equation}
where $\sigma(M)$ is the signature of $M$.)

Consider $\frac{1}{4p^2}b_2 = \ssum R(\eet, \eeth, J\eef, J\eefi)$.  As in the statement of Lemma~\ref{lemapp}, choose an orthonormal frame $\{e_2, Je_2, e_3, Je_3\}.$ The
$24$ permutations of $\{2,3,4,5\}$ break into three groups of eight permutations, such that \\
$\sgn(\sigma) R(\eet, \eeth, J\eef. J\eefi)$ is constant on each group.  For example, the permutations 
\begin{equation}\label{perm1}{\rm id}, (23), (23)(45), (45), (24)(35), (2534), (2435), (25)(34)
\end{equation}
have, for $\sigma = {\rm id}$, 
\begin{eqnarray*} \sgn(\sigma) R(\eet, \eeth, J\eef, J\eefi) &=& R(e_2, Je_2, Je_3, J^2e_3) = R(e_2, Je_2, Je_3,-e_3)\\
 & =& R(e_2, Je_2, e_3, Je_3),
 \end{eqnarray*}
and have, for $\sigma = (25)(34)$, 
\begin{eqnarray*} \sgn(\sigma) R(\eet, \eeth, J\eef, J\eefi) &=& R(Je_3, e_3, J^2e_2, Je_2)
= -R(e_2, Je_2, Je_3, e_3)\\
& =& R(e_2, Je_2, e_3, Je_3).
\end{eqnarray*}
In fact, all eight permutations in this group give 
\begin{equation}\label{perm2}\sgn(\sigma) R(\eet, \eeth, J\eef, J\eefi) = R(e_2, Je_2, e_3, Je_3).
\end{equation}
  The eight permutations consisting of the product of $(34)$ with the eight permutations in (\ref{perm1}) 
  (with $(34)$ concatenated on the right) give
  \begin{equation}\label{perm3} \sgn(\sigma) R(\eet, \eeth, J\eef, J\eefi) = 
  -R(e_2, e_3, e_2, e_3),
  \end{equation}
  and the eight permutations consisting of the product of $(354)$ with 
  the eight permutations in (\ref{perm1}) give
  \begin{equation}\label{perm4} \sgn(\sigma) R(\eet, \eeth, J\eef, J\eefi) = 
  -R(e_2, Je_3, e_2, Je_3).
  \end{equation}
  Combining (\ref{perm2}), (\ref{perm3}), (\ref{perm4}) gives
 \begin{equation}\label{b2}  b_2 =32p^2[R(e_2, Je_2, e_3, Je_3) 
 -R(e_2, e_3, e_2, e_3) -  R(e_2, Je_3, e_2, Je_3)].
 \end{equation}
 
For $b_3$, up to a factor the permutations ${\rm id}, (23), (45), (23)(45)$ contribute \\
$R(e_3, Je_3, e_\ell, Je_\ell)$, the permutations $(25)(34), (24)(35), (2534), (2435)$ contribute\\
$R(e_2, Je_2, e_\ell, Je_\ell)$, and all other permutations contribute zero.  Thus
$$ b_3 = 8p^2[R(e_2, Je_2, e_\ell, Je_\ell) +
 R(e_3, Je_3, e_\ell, Je_\ell)].
$$
Letting
$e_\ell$ run through $\{e_2, Je_2, e_3, Je_3\}$ gives
\begin{equation}\label{b3} b_3 = 16p^2[R(e_2, Je_2, e_2, Je_2) + 2R(e_2, Je_2, e_3, Je_3) + 
R(e_3, Je_3, e_3, Je_3)].
\end{equation}
Using the same grouping of permutations as for $b_3$, we get the same contribution as $b_3$, but with the roles of $e_2$ and $e_3$ switched:
$$b_4 = 8p^2[R(e_3, Je_3, e_\ell, Je_\ell) +
 R(e_2, Je_2, e_\ell, Je_\ell)].$$
 This is the same as $b_3$, so 
\begin{equation} \label{b4} b_4 = b_3= 
16p^2[R(e_2, Je_2, e_2, Je_2) + 2R(e_2, Je_2, e_3, Je_3) + 
R(e_3, Je_3, e_3, Je_3)].
\end{equation}

Combining (\ref{b2}), (\ref{b3}), (\ref{b4}) gives
\begin{eqnarray}\label{b234}
\lefteqn{b_2+b_3+b_4}\nonumber\\ 
&=& 32p^2[3R(e_2, Je_2, e_3, Je_3) 
 -R(e_2, e_3, e_2, e_3) -  R(e_2, Je_3, e_2, Je_3)\\
 &&\quad + R(e_2, Je_2, e_2, Je_2) + 
 R(e_3, Je_3, e_3, Je_3)].\nonumber 
\end{eqnarray}

Finally, up to a factor $b_5$ contributes one for each permutation in (\ref{perm1}) and zero otherwise, so
\begin{equation}\label{b5} b_5 
=192p^4.
\end{equation}

\bigskip
\noindent {\bf Summary:}  
The Chern-Simons form (\ref{expa}) reduces to calculating the term $A_1 = b_1 + \ldots + b_5.$  By
(\ref{b1}), (\ref{b234}), (\ref{b5}), this becomes
\begin{eqnarray*}\label{bsum}\lefteqn{a^*CS^W_{5,\gamma}(\bxi, e_2, Je_2, e_3, Je_3) }\\
&=&   \frac{3p^2}{5}    
 \left\{ 32\pi^2 p_1(\Omega)(e_2, Je_2, e_3, Je_3)
+ 32p^2[3 R(e_2, Je_2, e_3, Je_3) -R(e_2, e_3, e_2, e_3)   \right.\nonumber\\
&&\quad \left. 
-  R(e_2, Je_3, e_2, Je_3)  + R(e_2, Je_2, e_2, Je_2)+ R(e_3, Je_3, e_3, Je_3)]\right.\nonumber\\
 &&\quad \left.+ 192p^4 \right\}\nonumber
\end{eqnarray*}

\section{Topologies on $LM$}

Throughout the paper, we take Sobolev topologies on mapping spaces like $LM$ and $\diffm$ instead of the more standard Fr\'echet topology.  In this Appendix, we collect results to prove
 that topological results like Thm. \ref{bigthm} are independent of the choice of topology.

\begin{lem}  Let $L^sM$ denote the loops on $M$ of Sobolev class $s$, and let $L^FM$ denote $LM$ with the Fr\'echet topology.  For all $s>s'> 1/2$, $L^sM$ and $L^{s'}M$ 
are diffeomorphic smooth Hilbert manifolds.
Moreover, $L^FM$ is 
homotopy equivalent to $L^sM$.  ${\rm Diff}^s(M)$, ${\rm Diff}^{s'}(M)$ and  ${\rm Diff}^F(M)$ are homotopy equivalent in the similar notation.  

\end{lem}

The proof applies to $\maps^s(N,M)$ for any compact manifold $N$.

\begin{proof} It follows from \cite[Ch. 6(E)]{eells} that $L^sM$ and $L^{s'}M$ 
are diffeomorphic smooth Hilbert manifolds.

The inclusion $i:L^FM\to L^sM$ is continuous. 
Put a Riemannian metric on $M$.
  As in the proof of Lemma 3.3, given $\delta >0$ and $\gamma\in L^sM$, there exists a loop $\gamma'\in L^FM$ with $d_M(\gamma(\theta),\gamma'(\theta) )<\delta$ for all $\theta.$  
  The construction of $\gamma'$ in  \cite[Thm. 10.21]{lee} shows that the map $r:L^sM\to L^FM$, 
  $r(\gamma) = \gamma'$ is continuous.  
 We have $r\circ i\sim {\rm Id}_{L^FM}$ with the homotopy given by letting $\delta\to 0.$
 Similarly, 
$i\circ r\sim {\rm Id}_{L^sM}.$  Thus $L^sM, L^FM$ are 
homotopy equivalent.  

For ${\rm Diff}^s(M)$, let $i:L^s(M)\to L^{s'}(M)$ be the inclusion. $i$ is continuous with dense image.  $V = {\rm Diff}^{s'}(M)$ is open
in $L^{s'}M$, and $i^{-1}(V) = {\rm Diff}^s(M).$  By \cite[Cor. to Thm. 12]{palais}, $i$ is a homotopy equivalence.  The same proof works for $i:{\rm Diff}^F(M)\to {\rm Diff}^s(M).$
\end{proof}

\bibliographystyle{amsplain}
\bibliography{Paper}

\bigskip
\hfill \today \\

\end{document}